\documentclass[10pt,a4paper]{amsart}

\usepackage{amsmath,amssymb,amsthm,mathrsfs}
\usepackage{microtype}
\usepackage[hidelinks]{hyperref}
\usepackage{txfonts}
\usepackage{graphicx}
\usepackage{tikz-cd}

\newtheorem{theorem}{Theorem}[section]
\newtheorem{lemma}[theorem]{Lemma}
\newtheorem{proposition}[theorem]{Proposition}
\newtheorem{corollary}[theorem]{Corollary}

\theoremstyle{definition}
\newtheorem{definition}[theorem]{Definition}
\newtheorem{example}[theorem]{Example}

\theoremstyle{remark}
\newtheorem{remark}[theorem]{Remark}

\numberwithin{equation}{section}

\newcommand{\shf}[1]{\mathscr{#1}}
\newcommand{\is}[1]{\mathscr{I}_{#1}}
\newcommand{\ox}{\mathcal O_X}
\newcommand{\Pic}{\operatorname{Pic}}
\newcommand{\NS}{\operatorname{NS}}
\newcommand{\rk}[1]{\operatorname{rk}(#1)}
\newcommand{\IT}{\operatorname{IT}}
\newcommand{\Sym}{\operatorname{Sym}}
\newcommand{\Ext}{\operatorname{Ext}}
\newcommand{\FM}{\Phi_{\mathcal P}}
\newcommand{\CFM}{\Phi^H_{\mathcal P}}
\newcommand{\widehatstar}{\mathbin{\widehat{*}}}

\newcommand{\sshf}[1]{\mathcal{O}_{#1}}
\newcommand{\iso}{\simeq}
\newcommand{\ses}[3]{0\rightarrow#1\rightarrow#2\rightarrow#3\rightarrow{0}}

\newcommand{\paren}[1]{\left(#1\right)}

\def\ch{\ensuremath{\mathrm{ch}}}
\def\C{\ensuremath{\mathbb{C}}}
\def\R{\ensuremath{\mathbb{R}}}
\def\Z{\ensuremath{\mathbb{Z}}}
\def\Ker{\mathop{\mathrm{Ker}}\nolimits}
\def\Kn{\mathrm{K}_{\mathrm{num}}}
\def\cA{\ensuremath{\mathcal A}}

\def\Stab{\mathop{\mathrm{Stab}}\nolimits}
\def\Coh{\mathop{\mathrm{Coh}}\nolimits}
\def\Hom{\mathop{\mathrm{Hom}}\nolimits}

\def\Forg{\ensuremath{\mathrm{Forg}}}
\def\Db{\mathrm{D}^{b}}

\title[A Reider-Type Criterion]{Syzygies of Polarized Abelian Surfaces: A Reider-Type Criterion}
\author{Chunyi Li, Lei Song and Xiao Wang}
\date{}


\begin{document}

\begin{abstract}
Let $(X,L)$ be a polarized complex abelian surface with $L^2=2d$. We establish a Reider-type criterion for Property $N_p$. If $d\geq7$, then $L$ satisfies Property $N_0$ if and only if there is no elliptic curve $E\subseteq X$ with $L\cdot E\leq2$, with one explicitly described exception. If $p\geq1$ and $d\geq(p+2)^2+1$, then $L$ satisfies Property $N_p$ if and only if there is no elliptic curve $E\subseteq X$ with $L\cdot E\leq p+2$. The numerical bounds on $d$ are optimal for $p=0,1$. These results improve upon earlier work of K\"{u}ronya--Lozovanu, Ito, and Rojas. We also construct a polarized abelian surface whose basepoint-freeness threshold is irrational.
\end{abstract}

\maketitle

\section{Introduction}
\subsection{Main results}
Reider's celebrated theorem \cite[Theorem~1]{Reider88} asserts that if $X$ is a smooth complex projective surface and $L$ is a nef line bundle with $L^2\geq5$, then $K_X+L$ is basepoint-free unless there exists a curve $E\subseteq X$ satisfying one of the following conditions:
\begin{itemize}
    \item[(i)] $L\cdot E=0$ and $E^2=-1$;
    \item[(ii)] $L\cdot E=1$ and $E^2=0$.
\end{itemize}
The underlying principle is that, once $L$ has sufficiently large volume, any failure of positivity for $K_X+L$ should be explained by the existence of special subvarieties of small $L$-degree.

On an abelian surface, every nef line bundle $L$ with $L^2>0$ is ample, and $L^2$ is necessarily even. Reider's theorem therefore takes the following particularly simple form.

\begin{proposition}[Reider]
Let $(X,L)$ be a polarized abelian surface with $L^2\geq6$. Then the following conditions are equivalent:
\begin{itemize}
    \item[(i)] $L$ is basepoint-free;
    \item[(ii)] there is no elliptic curve $E\subseteq X$ such that $L\cdot E\leq1$.
\end{itemize}
\end{proposition}

The purpose of this paper is to establish a Reider-type theorem for syzygies on arbitrary complex abelian surfaces. Viewing basepoint-freeness as Property $N_{-1}$, we extend the preceding statement to Property $N_p$ for every $p\ge 0$. 

\begin{theorem}\label{main thm}
Let $p\geq0$ be an integer, let $(X,L)$ be a polarized complex abelian surface, and write $L^2=2d$.
\begin{itemize}
    \item[(1)] Suppose that $d\geq7$. Then the following conditions are equivalent:
    \begin{enumerate}
        \item[(i)] $L$ satisfies Property $N_0$;
        \item[(ii)] $X$ contains no elliptic curve $E$ such that $L\cdot E\leq2$.
    \end{enumerate}
    The only exception occurs when $L\iso M^{\otimes2}$ for a line bundle $M$ satisfying $M^2=4$ and such that $|M|$ has no fixed component. In this case, condition~(ii) holds, but $L$ does not satisfy Property $N_0$.
    \item[(2)] Suppose that $p\geq1$ and $d\geq(p+2)^2+1$. Then the following conditions are equivalent:
    \begin{enumerate}
        \item[(i)] $L$ satisfies Property $N_p$;
        \item[(ii)] $X$ contains no elliptic curve $E$ such that $L\cdot E\leq p+2$.
    \end{enumerate}
\end{itemize}
\end{theorem}

The implication (i)$\Rightarrow$(ii) in Theorem~\ref{main thm} was proved by K\"{u}ronya--Lozovanu \cite[Theorem~4.1]{KL19} when $d\geq2p+3$. Thus our main contribution is the converse implication (ii)$\Rightarrow$(i).

Using infinitesimal Newton--Okounkov bodies, K\"{u}ronya and Lozovanu \cite{KL19} first established such a Reider-type result under the assumption $d\geq\frac{5}{2}(p+2)^2$. Ito \cite{Ito18} improved this bound to $d\geq2(p+2)^2+1$, and later proved Property $N_p$ for a general polarized abelian surface of type $(1, d)$ when $d\geq(p+2)^2+p+3$ \cite{Ito23}. Under the additional ``almost Picard-rank-one'' hypothesis that $L^2$ divides $L\cdot D$ for every curve $D\subseteq X$, Rojas \cite{Rojas22} established Property $N_0$ for $d\geq7$ and Property $N_p$ for $d\geq(p+2)^2+1$ when $p\geq1$. Theorem~\ref{main thm} removes these additional generality and divisibility assumptions. Moreover, its numerical bounds are optimal for $p=0,1$; see Example~\ref{sharpness for p=0, 1}.

A central tool in our proof is the \emph{basepoint-freeness threshold} $\epsilon_1(l)$ introduced by Jiang--Pareschi \cite[Section~8]{JP20}; see Section~\ref{sec:cohrkfunctionandepsilon1} for a brief review. By work of Jiang--Pareschi and Caucci, the inequality $\epsilon_1(l)<1/(p+2)$ implies that $L$ satisfies Property $N_p$. The following theorem contains our main technical estimates.

\begin{theorem}\label{main thm2}
Let $(X,L)$ be a polarized abelian surface, write $L^2=2d$, and let $p\geq0$ be an integer. Suppose that $X$ contains no curve $E$ with $E^2=0$ and $L\cdot E\leq p+2$. Then:
\begin{itemize}
    \item[(1)] If $d>2(p+2)^2$, then $\epsilon_1(l)<\tfrac1{p+2}$.
    \item[(2)] If $d=2(p+2)^2$, then $\epsilon_1(l)<\tfrac1{p+2}$, unless $L\iso M^{\otimes(p+2)}$ for some ample line bundle $M$ with $M^2=4$, in which case $\epsilon_1(l)=\tfrac1{p+2}$.
   \item[(3)] If $p\geq1$ and $d\geq(p+2)^2+1$, then $L$ satisfies Property $N_p$.
\end{itemize}
\end{theorem}
Part~(1) is due to Ito; we include it for completeness. In the range $(p+2)^2<d\leq2(p+2)^2$, however, it can happen that $\epsilon_1(l)\geq1/(p+2)$; see Example~\ref{threshold can be large}. Consequently, the criterion of Jiang--Pareschi and Caucci does not apply throughout this range, and a different argument is required.

For primitive polarizations, the bound on $d$ ensuring $\epsilon_1(l)<1/(p+2)$ can be improved further.

\begin{theorem}\label{main thm3}
Let $(X,L)$ be a polarized abelian surface of type $(1,d)$, and let $p\geq1$. Suppose that $X$ contains no elliptic curve $E$ with $L\cdot E\leq p+2$. If $d\geq2(p+2)^2-2(p+2)+1$, then $\epsilon_1(l)<\tfrac1{p+2}$.
\end{theorem}

Gross--Popescu \cite[Conjecture~(a)]{GP98} conjectured that every very ample line bundle $L$ of type $(1,d)$ on an abelian surface has homogeneous ideal generated by quadrics and cubics when $d\geq9$. Building on work of Lazarsfeld \cite{Lazarsfeld90} and Fuentes Garc\'ia \cite{Garcia04} involving representations of the theta group, Agostini \cite{Agostini17} proved the conjecture for all $d\geq7$. As a byproduct of our method, we obtain a theta-group-free proof of Agostini's theorem. Theorem~\ref{main thm} also shows that, when $d\geq10$, cubic generators are necessary whenever $X$ contains an elliptic curve $E$ with $L\cdot E\leq3$.

\begin{proposition}\label{GP Conjecture}
Let $(X, L)$ be a polarized abelian surface of type $(1, d)$ such that
$L$ is very ample and $d\ge 7$. Then $L$ is projectively normal, and its
homogeneous ideal is generated by quadrics and cubics. Moreover, if
$d\geq 10$, then the homogeneous ideal is generated by quadrics if and
only if the abelian surface does not contain any elliptic curve $E$ with
$L\cdot E\leq 3$.
\end{proposition}

\subsection{Outline of the proofs}

We give an overview of the proofs of Theorems~\ref{main thm2} and~\ref{main thm3}.

The first input is the interpretation, due to Lahoz--Rojas \cite{LR23}, of $\epsilon_1(l)$ in terms of the Harder--Narasimhan filtration of the ideal sheaf $\is{q}$, for $q\in X$, with respect to the Bridgeland stability conditions $\sigma_{s,t}$ for $s<0$ and $0<t\ll1$; see Section~\ref{sec:stabonabsurf} for their definition. If $\is{q}$ crosses no actual wall in the region $s<0$, then $\epsilon_1(l)=1/\sqrt d$. Otherwise, let $\shf F$ and $\shf Q$ denote the destabilizing subobject and quotient on the actual wall. One of our key observations is that both $\shf F$ and $\shf Q$ have discriminant zero, where we use the ordinary discriminant rather than the polarized one, and are stable throughout a suitable slice of the Bridgeland stability manifold. Consequently, the Harder--Narasimhan filtration of $\is{q}$ for the weak stability function $\nu_{s,0}$ has two graded pieces, and $\epsilon_1(l)$ can be expressed explicitly in terms of their Chern characters; see Proposition ~\ref{prop:1stwall}. The  formula enables us to construct a polarized abelian surface with irrational $\epsilon_1(l)$; see Example \ref{irrational threshold}.

The object $\shf F$ is always a vector bundle. We call $r=\rk{\shf {F}}$ the \emph{rank of the wall}. If $r=1$, then $\shf F\iso\mathcal O_X(-C)$ for a divisor $C$ with $C^2=0$, and
\[
\epsilon_1(l)=\frac1{L\cdot C}.
\]
The assumed absence of elliptic curves of small $L$-degree therefore immediately yields $\epsilon_1(l)<1/(p+2)$.

Consider now the higher-rank case $r\geq2$. The wall gives rise to a short exact sequence of coherent sheaves
\begin{equation}\label{eq:1.1}
\ses{\shf G}{\shf F}{\is{q}},
\end{equation}
where $\shf F$ and $\shf G$ are simple semihomogeneous vector bundles and $C=-\ch_1(\shf F)$ is effective.

If, nevertheless, $\epsilon_1(l)\geq1/(p+2)$, then
\[
L\cdot C\leq(p+2)r+\frac{(r-1)d}{p+2}.
\]
Under the stronger assumption
\[
\tfrac1{p+1}>\epsilon_1(l)\geq\tfrac1{p+2},
\]
this inequality becomes the numerical identity
\begin{equation}\label{eq:1.2}
L\cdot C=(p+2)r+\left\lfloor\frac{(r-1)d}{p+2}\right\rfloor.
\end{equation}

Combining \eqref{eq:1.2} with the Hodge index theorem and arguing by contradiction and induction on $p$ yields the required threshold estimate for $p=0$, apart from the stated exception, whenever $d>2(p+2)^2$, and, for primitive polarizations, whenever $d>2(p+2)^2-2(p+2)$. This proves Theorem~\ref{main thm3}.

For arbitrary polarizations with $d\leq2(p+2)^2$, however, the numerical restriction alone does not force $\epsilon_1(l)<1/(p+2)$. Indeed, the reverse inequality can occur; see Example~\ref{threshold can be large}. Thus the basepoint-freeness-threshold criterion is not sufficient by itself to prove Theorem~\ref{main thm2}(3).

To treat small $d$ uniformly, the next key step is to apply the Fourier--Mukai transform to \eqref{eq:1.1} and then pull back by the polarization isogeny $\varphi_l$, more precisely, to apply the skew Pontryagin product $L\widehatstar \bullet$; see \S ~\ref{subsect:spp}. This produces a short exact sequence 
\begin{equation}\label{eq:1.3}
0\longrightarrow \mathcal A\longrightarrow \mathcal B
\longrightarrow M_L\longrightarrow 0,
\end{equation}
where $\mathcal A$ and $\mathcal B$ are certain semihomogeneous vector bundles and $M_L$ is the kernel bundle of $L$. To prove Theorem~\ref{main thm2}(3), we will apply Proposition \ref{prop:green}. To this end, for every integer $1\le k\le p+1$,  consider the Schur complex associated with \eqref{eq:1.3}, which is indeed exact,
\begin{equation*}\label{eq:intro-schur-complex}
\begin{split}
0\longrightarrow S^k\mathcal A
&\longrightarrow S^{k-1}\mathcal A\otimes\mathcal B
\longrightarrow S^{k-2}\mathcal A\otimes\bigwedge^2\mathcal B
\longrightarrow\cdots\\
&\longrightarrow\mathcal A\otimes\bigwedge^{k-1}\mathcal B
\longrightarrow\bigwedge^k\mathcal B
\longrightarrow\bigwedge^kM_L\longrightarrow0,
\end{split}
\end{equation*}
where $S^k$ and $\bigwedge^k$ denote the $k$-th symmetric and exterior powers, respectively. It suffices to prove the following two vanishings, for $1\leq k\leq p+1$ and $h\geq 1$:
\begin{equation}\label{eq:intro-vanishings}
H^1\big(X,\bigwedge^k\mathcal B\otimes L^{\otimes h}\big)=0
\quad\text{and}\quad
H^2\big(X,\mathcal A\otimes\bigwedge^{k-1}\mathcal B\otimes L^{\otimes h}\big)=0.
\end{equation}

Assuming now that $\epsilon_1(l)\ge \frac{1}{p+2}$, Proposition \ref{positivity of B term} shows that the $\mathbb Q$-twisted sheaf $\mathcal B\langle\frac{1}{p+1}l\rangle$ is $\IT(0)$, which gives the first vanishing in \eqref{eq:intro-vanishings}.  In contrast, $\mathcal A\langle\frac{1}{p+1}l\rangle$ is not $\IT(0)$, so the analogous argument does not yield the $H^2$-vanishing.

The second vanishing requires a different observation: the bundles $\mathcal A\otimes\bigwedge^{k-1}\mathcal B\otimes L^{\otimes h}$ are semihomogeneous, although not necessarily simple. Outside three exceptional numerical cases, a uniform estimate gives a positive $L$-degree;  semistability and Serre duality then give the desired $H^2$-vanishing. The exceptional cases are
\begin{equation*}\label{eq:intro-exceptions}
(p,d,r,L\cdot C)=(1,10,2,9),\quad(1,12,2,10),\quad(2,17,2,12).
\end{equation*}
In each one, the normalized first Chern class has a negative square, so the index lemma for semihomogeneous vector bundles (Lemma~\ref{lem:index}) again gives the $H^2$-vanishing.

The restriction $p\geq 1$ matters at the bound $d\geq (p+2)^2+1$. Indeed, if $A$ is a general polarization of type $(1,2)$ and $L=A^{\otimes 2}$, then $d=8$, the surface $X$ contains no elliptic curve, and $L$ is not projectively normal; see \cite[Lemma~6]{Ohbuchi93}. Thus the unrestricted statement for arbitrary nonprimitive polarizations already fails when $p=0$.

While this paper was being written, we learned that Atsushi Ito had independently discovered the irrational basepoint-freeness threshold appearing in Example~\ref{irrational threshold}.

\vspace{5pt}

\noindent\textbf{Organization of the paper.} Section~\ref{sec:prelim} collects the necessary background on Property $N_p$, cohomological rank functions, $\mathbb{Q}$-twisted sheaves, semihomogeneous bundles, and Fourier--Mukai transforms. Section~\ref{sec:Bridgeland} reviews the relevant Bridgeland stability conditions and establishes the stability properties of objects with zero discriminant. In Section~\ref{sec:first-wall}, we determine the actual wall of $\is{q}$ and its stable factors. Section~\ref{sec:numerical-wall} derives some numerical restrictions and proves the threshold estimates, including Theorem~\ref{main thm3} and the case $p=0$. Section~\ref{sec:mlcompute} constructs a Fourier--Mukai presentation of $M_L$ and establishes two cohomology vanishings required in the proof. Section~\ref{sec:mainproof} applies the Schur complex to complete the proofs of the main results. Finally, Section~\ref{sec:examples} presents examples and complementary results, including an irrational basepoint-freeness threshold and an upper bound for the threshold.

\vspace{5pt}

\noindent\textbf{Conventions.} Throughout the paper, we work over $\mathbb C$. We write $\Db(X)=\Db(\Coh(X))$ for the bounded derived category of coherent sheaves on $X$. For a coherent sheaf $\shf F$ on $X$, the notation $H^i(X,\shf F)$ denotes sheaf cohomology and $\ch_i(\shf{F})$ denotes its $i$-th Chern character. For  an object $\shf F\in\Db(X)$, the notation $\mathcal{H}^i(\shf F)$ denotes the $i$-th cohomology sheaf. By an elliptic curve on an abelian surface, we mean a smooth curve of genus one. We identify divisors with their numerical or cohomological classes when no confusion is likely. 
\vspace{5pt}

\noindent\textbf{AI Disclosure.} The project was initiated in late 2024, when large language models (LLMs) were not yet substantially involved in mathematical research. The two main conceptual innovations, namely, the discriminant argument that avoids the `almost Picard number one' assumption in \cite{LR23}, and the presentation of the kernel bundle via the destabilizing sequences of ideal sheaves, were developed entirely by the authors, without any LLM assistance. Following the release of ChatGPT 5.6 Pro in July 2026, we used it to assist with some of the computations in Section~\ref{sec:mlcompute}, enabling us to identify the three exceptional cases rapidly. The use of LLM accelerated the final stages of the project.

\vspace{5pt}

\noindent\textbf{Acknowledgements.} During the preparation of the article, we benefited a lot from discussions and conversations with Lawrence Ein, Mart\'i Lahoz, Robert Lazarsfeld, Andr\'es Rojas and Hao Max Sun. C.~L. was supported by the Royal Society URF/R1/201129 and 251025, \emph{Stability Conditions and Applications in Algebraic Geometry}. L.~S. was supported by NSFC grant No.~12471043 and  Guangdong Basic and Applied Basic Research Foundation No.~ 2025A1515012258.

\section{Preliminaries}\label{sec:prelim}
We collect the definitions and facts used in the sequel.

\subsection{Property \texorpdfstring{$N_p$}{Np} and the kernel bundle}

Let $X$ be a projective variety embedded in $\mathbb P^N$ by the complete linear system $|L|$. Put
\[
S=\Sym\paren{H^0(X,L)}, \qquad R(X,L)=\bigoplus_{h\geq 0}H^0(X,L^{\otimes h}).
\]
The section ring $R(X,L)$ is naturally a graded $S$-module. The line bundle $L$ is said to  satisfy  \textit{Property $N_0$} if the natural map $S\to R(X,L)$ is surjective. For an integer $p\geq 1$, it is said to satisfy \textit{Property $N_p$} if it satisfies Property $N_0$ and the first $p$ steps in the minimal graded free resolution of $R(X,L)$ over $S$ are linear; see
\cite[Section~1.8.D]{LazarsfeldI}.

Let $M_L$ be the kernel bundle defined by the evaluation map
\begin{equation*}\label{eq:evaluation}
0\longrightarrow M_L\longrightarrow H^0(X,L)\otimes\ox
\longrightarrow L\longrightarrow 0.
\end{equation*}
We shall use the following cohomological criterion for Property $N_p$; see, for example, \cite[Proposition 2.5]{AN10}.
\begin{proposition}\label{prop:green}
Let $L$ be a globally generated line bundle on a projective
variety $X$ over a field of characteristic zero. If
\[
H^1\big(X,\bigwedge^kM_L\otimes L^{\otimes h}\big)=0
\]
for every $1\leq k\leq p+1$ and $h\geq 1$, then $L$ satisfies Property
$N_p$. If $H^1(X, L^{\otimes i})=0$ for all $i\ge 1$, then the converse is also true. 
\end{proposition}

\subsection{Cohomological rank function and basepoint-freeness threshold}\label{sec:cohrkfunctionandepsilon1}
Let $(X,L)$ be a polarized abelian variety of arbitrary dimension, and let $l$ denote the class of $L$ in the N\'eron--Severi group $\NS(X)$.

Jiang--Pareschi \cite{JP20} introduced the \emph{cohomological rank functions} $h^i_{\shf F,l}(x)$ for $\shf F\in D^b(X)$ and $x\in\mathbb R$. Roughly speaking, the value $h^i_{\shf F, l}(x)$ interpolates the dimension of $H^i(X,\shf F\otimes L^{\otimes x}\otimes\alpha)$ for general $\alpha\in\widehat X\iso\Pic^0(X)$, even though this cohomology group is not defined unless $x\in\mathbb Z$. 

In particular, when $\shf F$ is the ideal sheaf $\is q$ of a point $q\in X$, Jiang--Pareschi \cite[Section~8]{JP20} introduced the \emph{basepoint-freeness threshold}
\[
\epsilon_1(l)=\inf\big\{x\in\mathbb{Q} \:|\: h^1_{\is{q},l}(x)=0\big\},
\]
which encodes positivity properties of $l$. It is independent of the point $q$ and satisfies $\epsilon_1(l)\le 1$.  We use the following criterion of Jiang--Pareschi \cite{JP20} and Caucci \cite{Caucci20}.
\begin{theorem}\label{thm:caucci}
A line bundle $L$ is basepoint-free if and only if $\epsilon_1(l)<1$. Given an integer $p\ge 0$, if $\epsilon_1(l)<\frac{1}{p+2}$, then $L$ satisfies Property $N_p$.
\end{theorem}
Thus bounds on $\epsilon_1(l)$ give an effective approach to syzygies of ample line bundles on abelian varieties.

\subsection{IT(0) and GV of $\mathbb{Q}$-twisted sheaves}
Let $X$ be an abelian variety. A coherent sheaf $\shf F$ on $X$ is said to satisfy the \emph{Index Theorem with index $i$}, abbreviated $\IT(i)$, if
\[
H^j(X,\shf F\otimes\alpha)=0
\qquad
\text{for every }\alpha\in\Pic^0(X)\text{ and every }j\neq i.
\]
The notion of generic vanishing (GV) sheaf was developed by Pareschi--Popa \cite{PP11}. The corresponding notions for $\mathbb{Q}$-twisted coherent sheaves on $X$ were introduced by Jiang--Pareschi \cite{JP20}, and further developed by Caucci \cite{Caucci20} and Ito \cite{Ito22}.  

Let $l\in \NS(X)$ be a polarization, and $x\in \mathbb{Q}$. The symbol $\shf{F}\langle xl\rangle $ denotes the equivalence class of the pair $(\shf{F}, xl)$, where the equivalence relation is given by
\[(\shf{F}\otimes L^{m}, xl)\sim (\shf{F}, (m+x)l)\]
for any line bundle $L$ with $[L]=l$ and $m\in \mathbb{Z}$. In particular, $(\shf{F}\otimes P_{\alpha})\langle xl \rangle=\shf{F}\langle xl \rangle$ for any $\alpha\in \widehat{X}$. 
\begin{definition}
A $\mathbb{Q}$-twisted sheaf $\shf{F}\langle xl \rangle$ is  IT(0) (resp. GV) if $\mu^*_b\shf{F}\otimes L^{\otimes ab}$ is IT(0) (resp. GV) in the usual sense, where $x=\frac{a}{b}$ with $b>0$, the line bundle $L$ represents $l$, and $\mu_b: X\rightarrow X$ is multiplication by $b$.
\end{definition}
The definition does not depend on the choice of $a, b\in \mathbb{Z}$ and of $L\in \Pic(X)$. 

\begin{proposition}[{\cite[Proposition 3.4]{Caucci20}}]\label{positivity of tensor}
Let $\shf{F}, \shf{G}$ be coherent sheaves on $X$. Assume that one of them is locally free. If $\shf{F}\langle xl\rangle$ is IT(0) and $\shf{G}\langle yl\rangle$ is GV, then $\shf{F}\langle xl\rangle\otimes \shf{G}\langle yl\rangle:=(\shf{F}\otimes\shf{G})\langle (x+y)l\rangle$ is IT(0).
\end{proposition}

For a line bundle $M$, the $\mathbb{Q}$-twisted sheaf $M\langle xl\rangle$ is IT(0) (resp. GV) if and only if  $M+xL$ is ample (resp. nef).

\subsection{Semihomogeneous bundles}\label{subsec: semihomogeneous}
A vector bundle $\shf E$ on an abelian variety is called \emph{semihomogeneous} if, for every $q\in X$, there exists $\alpha_q\in\Pic^0(X)$ such that
\[
t_q^*\shf E\simeq\shf E\otimes\alpha_q,
\]
where $t_q: X\rightarrow X$ is the translation of $X$ by $q$. 

We set
\[
\delta(\shf E)=\frac{\ch_1(\shf E)}{\rk{\shf E}}.
\]
For a semihomogeneous bundle $\shf{E}$, one has
\begin{equation}\label{eq:2.1}
\ch(\shf E)=\rk{\shf E}\,e^{\delta(\shf E)}.
\end{equation}
Semihomogeneous bundles are Gieseker-semistable, in particular, slope-semistable with respect to every polarization \cite[Proposition~6.13]{Mukai78}.

We shall repeatedly use the following elementary properties. If $\shf E$ and $\shf F$ are semihomogeneous, then every nonzero bundle of the form $\shf E\otimes\shf F$ and $\bigwedge^j\shf E$ is semihomogeneous. Indeed,
\[
t_q^*(\shf E\otimes\shf F)
\simeq(\shf E\otimes\shf F)\otimes
\alpha_q\otimes\beta_q
\]
and
\[
t_q^*\big(\bigwedge^j\shf E\big)
\simeq\bigwedge^j\shf E\otimes\alpha_q^{\otimes j}.
\]
Moreover,
\begin{equation}\label{eq:2.2}
\delta(\shf E\otimes\shf F)=\delta(\shf E)+\delta(\shf F),
\qquad
\delta\big(\bigwedge^j\shf E\big)=j\delta(\shf E).
\end{equation}

We record the index lemma for semihomogeneous vector bundles in the form needed below.
\begin{lemma}\label{lem:index}
Let $\shf E$ be a semihomogeneous vector bundle on an abelian surface.
\begin{enumerate}
\item[(i)] If $\delta(\shf E)$ is ample, then $\shf E$ is $\IT(0)$.
\item[(ii)] If $\delta(\shf E)^2<0$, then $\shf E$ is $\IT(1)$. In particular,
\[
H^2(X,\shf E\otimes\alpha)=0
\qquad\text{for every }\alpha\in{\Pic}^0(X).
\]
\end{enumerate}
\end{lemma}

\begin{proof}
Suppose first that $\shf E$ is simple. A nondegenerate simple semihomogeneous bundle satisfies the Index Theorem, and its index is the index of its determinant. This follows from Mukai's isogeny description and Mumford's Index Theorem; see \cite[Proposition~6.3]{Gulbrandsen08} and \cite[\S ~16]{Mumford74}. If $\delta(\shf E)$ is ample, then $\det\shf E$ is ample and has index zero. This proves (i) when $\shf E$ is simple.

If $\delta(\shf E)^2<0$, then $(\ch_1(\det\shf E))^2=\rk{\shf E}^{2}\delta(\shf E)^2<0. $ Thus $\chi(\det\shf E)<0$.  The index of a nondegenerate line bundle on an abelian surface belongs to $\{0,1,2\}$, and the sign of its Euler characteristic is $(-1)^i$, where $i$ is the index. It follows that $i(\det\shf E)=1$. Hence $i(\shf E)=1$, which proves (ii) in the simple case.

For an arbitrary semihomogeneous bundle, Mukai's structure theorem \cite[Proposition~6.18]{Mukai78} gives a decomposition into summands, each of which admits a filtration with the same simple semihomogeneous factors. Applying the simple case to these factors and then using the associated long exact cohomology sequences proves both assertions in general.
\end{proof}

The following positivity criterion is elementary.

\begin{lemma}\label{lem:positivecone}
Let $(X,L)$ be a polarized abelian surface and let
$D\in\NS(X)_{\mathbb R}$.
If $D^2>0$ and $D\cdot L>0$, then $D$ is ample.
\end{lemma}

\begin{proof}
The hypotheses imply that $D$ is in the component of the positive cone that contains $L$. For any irreducible curve $B$, $B^2\ge 0$ and $B\cdot L>0$, so the class of $B$ is a nonzero element of the closure of the component. The Hodge index theorem then gives $D\cdot B>0$. Hence $D$ is ample by the Nakai--Moishezon criterion. 
\end{proof}

\subsection{Fourier--Mukai transforms and skew Pontryagin products}\label{subsect:spp}

Let $\mathcal P$ be a normalized Poincar\'e bundle on
$X\times\widehat X$, and write
\[
\FM:D^b(X)\longrightarrow D^b(\widehat X)
\]
for the corresponding Fourier--Mukai transform. The Skew Pontryagin product was introduced by Pareschi in \cite{Pareschi00} and used extensively in \cite{PP04} to study multiplication maps on spaces of sections of line bundles on abelian varieties.

Let $L$ be an ample line bundle and $\varphi_l:X\to\widehat X$ be the polarization isogeny. For any coherent sheaf $\shf{E}$ on $X$, the skew Pontryagin product satisfies 
\begin{equation}\label{eq:skewFM}
L\widehatstar\shf E
\simeq L\otimes\varphi_l^*\FM(L\otimes\shf E)
\end{equation}
whenever $\FM(L\otimes\shf E)$ is a vector bundle; see \cite[Remark 1.2]{Pareschi00} and \cite[(3.10)]{Mukai81}.  In particular, at the origin one has 
\begin{equation}\label{eq:skewkernel}
L\widehatstar\is 0\simeq M_L,
\end{equation}
see \cite[Proof of Prop. 8.1]{JP20}. For a general point $q$, the bundle $L\widehatstar\is q$ is a translate of $M_L$, up to a topologically trivial twist.  Since $\epsilon_1(l)$ is independent of $q$, we take $q=0$ whenever we use \eqref{eq:skewkernel}.  

Fourier--Mukai transforms preserve semihomogeneity whenever the transform is locally free. Indeed, by
\cite[Proposition~5.1 and Definition~5.2]{Mukai78}, a vector bundle $\shf{E}$ is semihomogeneous if and only if 
\[
\dim\left\{(x,\alpha)\in X\times \widehat X \:|\: t_x^*\shf E\simeq\shf E\otimes P_\alpha\right\}
=\dim X.
\]
The covariance identities
\[
\FM(t_x^*\shf E)\simeq\FM(\shf E)\otimes P_{\pm x},
\qquad
\FM(\shf E\otimes P_\alpha)\simeq t_{\pm\alpha}^*\FM(\shf E)
\]
identify the corresponding translation--twist groups.  Pullback by an isogeny preserves semihomogeneity by \cite[Proposition~5.4(1)]{Mukai78}. Consequently, if $L\otimes\shf E$ is $\IT(0)$ and semihomogeneous, then $L\widehatstar\shf E$ is a semihomogeneous vector bundle.

We conclude the section with a useful calculation.

\begin{lemma}\label{lem:FM-calculation}
Let $(X, L)$ be a polarized abelian surface and write $L^2=2d$. Let $\shf E$ be a semihomogeneous vector bundle on $X$ with
\[
\ch(\shf E)=(n,-C,u),
\]
and assume that $L\otimes\shf E$ is $\IT(0)$.  Put
$\mathcal W_{\shf E}=L\widehatstar\shf E$. Then the following hold:
\begin{align*}
\rk{\mathcal W_{\shf E}}&=nd-L\cdot C+u,\\
c_1(\mathcal W_{\shf E})&=ul-dC,\\
\delta(\mathcal W_{\shf E})
&=\frac{ul-dC}{nd-L\cdot C+u}.
\end{align*}
\end{lemma}

\begin{proof}
Since $L\otimes\shf E$ is $\IT(0)$, isomorphism \eqref{eq:skewFM} and the 
Riemann--Roch theorem give
\[
\rk{\mathcal W_{\shf E}}=\rk{\FM(L\otimes\shf E)}
=\chi(L\otimes\shf E)=nd-L\cdot C+u.
\]
Put $D:=c_1(L\otimes \shf{E})=nl-C$. Let $\CFM$ be the cohomological Fourier--Mukai transform, and write
\[\CFM(D)=-\widehat D.\]
We shall use the identity
\begin{equation}\label{eq:pullback-isogeny}
\varphi_l^*(\widehat D)=(L\cdot D)l-dD.
\end{equation}
Granting this, it follows that 
\[\begin{aligned}
c_1(\mathcal W_{\shf E})&=\rk{\mathcal W_{\shf E}}l-\varphi_l^*(\widehat D)\\
&=(nd-L\cdot C+u)l
-\left((2nd-L\cdot C)l-d(nl-C)\right)\\
&=ul-dC,
\end{aligned}\]
as desired, and the formula for $\delta(\mathcal W_{\shf E})$ follows from its definition.

Turning to the proof of \eqref{eq:pullback-isogeny}, let $m: X\times X\rightarrow X$ be the addition map, $q_i: X\times X\rightarrow X$ be the $i$-th projection, and $p_X, p_{\widehat{X}}$ be the projections from $X\times \widehat{X}$ to the corresponding factors. Consider the Cartesian diagram
\[
\begin{tikzcd}
X\times X
  \arrow[r, "\operatorname{id}_X\times\varphi_l"]
  \arrow[d, "q_2"']
&
X\times\widehat X
  \arrow[d, "p_{\widehat X}"]
\\
X
  \arrow[r, "\varphi_l"']
&
\widehat X.
\end{tikzcd}
\]
Let $\xi=c_1(\mathcal{P})$. By \cite[p.~161, proof of (3.10)]{Mukai81}, we have
\[
(\operatorname{id}_X\times\varphi_l)^*\mathcal P
\simeq
m^*L\otimes(L^{-1}\boxtimes L^{-1}).
\]
Consequently, $\eta:=
(\operatorname{id}_X\times\varphi_l)^*\xi=m^*l-q_1^*l-q_2^*l$.

By the definition of $\CFM$ and flat base change, we have
\begin{equation}\label{eq:FM-pullback}
\begin{aligned}
-\varphi_l^*(\widehat D)&=\varphi^*_l\CFM(D)\\
&=\varphi^*_l {p_{\widehat{X}}}_*\paren{p^*_XD\cdot \exp(\xi)}\\
&={q_2}_*\paren{\operatorname{id}_X\times\varphi_l}^*\paren{p^*_XD\cdot \exp(\xi)}\\
&= {q_2}_*\left(q_1^*D\cdot\exp(\eta)\right)\\
&={q_2}_*\left(q_1^*D\cdot\frac{\eta^2}{2}\right).
\end{aligned}
\end{equation}
For the last equality, note that under the K\"unneth decomposition, \(q_1^*D\) has bidegree \((2,0)\), whereas \(\eta\) has bidegree
\((1,1)\). Thus \(q_1^*D\cdot\eta^j\) has degree \(2+j\) on the
first factor, hence only $j=2$ contributes. 

Finally, to evaluate the right-hand side of \eqref{eq:FM-pullback}, for any \(\alpha\in\NS(X)\), applying
the projection formula, we get
\[
\begin{aligned}
\left(
{q_2}_*\left(q_1^*D\cdot\frac{\eta^2}{2}\right)
\right)\cdot\alpha
&=
\frac12
q_1^*D\cdot q_2^*\alpha\cdot
\left(m^*l-q_1^*l-q_2^*l\right)^2\\
&=d(D\cdot\alpha)
 -(L\cdot D)(L\cdot\alpha),
\end{aligned}
\]
which shows 
\[
-\varphi_l^*(\widehat D)=dD-(L\cdot D)l.
\]
This proves \eqref{eq:pullback-isogeny} and completes the proof.
\end{proof}

\section{Bridgeland stability on abelian surfaces}\label{sec:Bridgeland}
\subsection{Stability conditions}
Let $X$ be a smooth complex projective variety. Write $\Db(X)$ for its bounded derived category of coherent sheaves and $\Kn(X)$ for its numerical Grothendieck group. We briefly recall the definition and some basic properties of Bridgeland stability conditions on $\Db(X)$; see \cite{Bridgeland07}.

\begin{definition}
Let $\cA$ be an abelian category.  A \emph{stability function} on $\cA$ is a group homomorphism
\[
Z\colon K_0(\cA)\longrightarrow\C
\]
such that, for every nonzero object $E\in\cA$, either $\Im Z(E)>0$, or $\Im Z(E)=0$ and $\Re Z(E)<0$.
\end{definition}

For a nonzero object $E\in\cA$, its \emph{slope} with respect to $Z$ is defined by
\[
\mu_Z(E)=\begin{cases} -\dfrac{\Re Z(E)}{\Im Z(E)}, & \Im Z(E)>0,\\ +\infty, & \Im Z(E)=0. \end{cases}
\]
We say that $E$ is $\mu_Z$-\emph{(semi)stable} if, for every nonzero proper subobject $F\subset E$,
\[
\mu_Z(F)<(\leq)\;\mu_Z(E/F).
\]

\begin{definition}
Let $\cA$ be the heart of a bounded $t$-structure on $\Db(X)$, and let $Z\colon\Kn(X)\to\C$ be a stability function on $\cA$. Denote by $\sigma$ the pair of data $(\cA,Z)$.

A nonzero object $E\in\Db(X)$ is called $\sigma$-\emph{(semi)stable} if there exists $n\in\Z$ such that $E[n]\in\cA$ and $E[n]$ is $\mu_Z$-(semi)stable.

We call $\sigma$ a \emph{stability condition} on $\Db(X)$ if the following two properties hold.
\begin{enumerate}
\item Every nonzero object $E\in\cA$ admits a Harder--Narasimhan (HN) filtration
\begin{equation*}
0=E_0\subset E_1\subset\cdots\subset E_m=E
\end{equation*}
whose factors $A_i=E_i/E_{i-1}$ are $\sigma$-semistable and satisfy
\[
\mu_Z(A_1)>\mu_Z(A_2)>\cdots>\mu_Z(A_m).
\]

\item (Support Property) There exists a quadratic form $Q$ on $\Kn(X)_{\mathbb R}:=\Kn(X)\otimes \R$ such that
\begin{itemize}
\item $Q|_{\Ker Z}$ is negative definite;
\item $Q([E])\geq0$ for every $\sigma$-semistable object $E$.
\end{itemize}
\end{enumerate}
\end{definition}

Every strictly $\sigma$-semistable object admits a finite filtration whose factors are $\sigma$-stable objects of the same slope.  These factors are called its \emph{Jordan--H\"older factors}; see \cite[Appendix~A]{BMS16}.

We denote by $\Stab(X)$ the space of numerical stability conditions on $\Db(X)$.  It carries a natural Hausdorff topology, with respect to which the locus where a fixed object is stable is open.  Bridgeland's Deformation Theorem \cite{Bridgeland07} asserts that the forgetful map
\[
\Forg\colon\Stab(X)\longrightarrow\Hom(\Kn(X),\C),
\qquad
(\cA,Z)\longmapsto Z,
\]
is a local homeomorphism. In particular, the space $\Stab(X)$ is a complex manifold of dimension $\rk{\Kn(X)}$.

We shall use the following effective form of the deformation theorem; see \cite[Propositions~A.5 and A.8]{BMS16} and also \cite{Bayer19}.

\begin{proposition}[{\cite[Propositions~A.5 and A.8]{BMS16}}]
\label{propdef:stabQsigma}
Let $\sigma=(\cA,Z)\in\Stab(X)$, and suppose that the support property for $\sigma$ is given by a quadratic form $Q$.
Let
\[
\mathcal U_Q
=
\left\{
Z'\in\Hom(\Kn(X),\C):
Q|_{\Ker Z'}\text{ is negative definite}
\right\},
\]
and denote by $W=W(Q,Z)$ the connected component of $\mathcal U_Q$ containing $Z$.

Let
\[
\Stab(X,Q,\sigma)\subset\Stab(X)
\]
be the connected component of $\Forg^{-1}(W)$ containing $\sigma$. Then:
\begin{enumerate}
\item the restriction $\Forg|_{\Stab(X,Q,\sigma)} \colon \Stab(X,Q,\sigma)\longrightarrow W$ is a covering map;

\item every $\sigma'\in\Stab(X,Q,\sigma)$ satisfies the support property with respect to the same quadratic form $Q$;

\item if $E\in\Db(X)$ is $\sigma$-stable and $Q([E])=0$, then $E$ is $\sigma'$-stable for every $\sigma'\in\Stab(X,Q,\sigma)$.
\end{enumerate}
\end{proposition}
\subsection{Stability conditions on an abelian surface}\label{sec:stabonabsurf}

Let $(X,L)$ be a polarized complex abelian surface. We recall some basic facts and the standard construction of a real two-dimensional family of stability conditions obtained by tilting $\Coh(X)$; see \cite[Section~6]{Bridgeland:K3} and \cite{AB12}. 

For $v=(r,D,u), w=(r',D',u')$ in $\Kn(X)$, define the Mukai pairing and the discriminant by
\begin{equation*}
\langle v,w\rangle = D\cdot D'-ru'-r'u, \qquad \Delta(v) = \langle v,v\rangle = D^2-2ru.
\end{equation*}
By Riemann--Roch,
\[
\chi(E,F)=-\langle [E],[F]\rangle, \qquad  \Delta(E)=-\chi(E,E).
\]
Since the intersection form on an abelian surface is even, $\Delta(E)$ is an even integer.

We also record a parity observation that will be useful below. Suppose that $E\in\Db(X)$ is stable with respect to a Bridgeland stability condition. Then stability gives
\[
\Hom(E,E)\simeq\mathbb C, \qquad \text{ and } \Hom(E,E[k])=0 \qquad\text{for }k<0.
\]
By Serre duality, we have 
\[
\chi(E,E)=2-\dim\Ext^1(E,E).
\]
Since $\chi(E,E)=-\Delta(E)$ is even, it follows that
\begin{equation}\label{eq:ext1even}
\dim\Ext^1(E,E) \quad\text{is even}.
\end{equation}

For $s\in\mathbb R$, define the $L$-slope of a coherent sheaf $E$ of positive rank by
\[
\mu_L(E) = \frac{L\cdot\ch_1(E)}{L^2\ch_0(E)},
\]
and set $\mu_L(E)=+\infty$ for every nonzero torsion sheaf. Define
\begin{align*}
\Coh_L^{>s}(X) &:= \left\{ E\in\Coh(X):  \mu_L(F)>s \text{ for every nonzero quotient }E\twoheadrightarrow F \right\},\\
\Coh_L^{\leq s}(X) &:= \left\{ E\in\Coh(X): \mu_L(F)\leq s \text{ for every nonzero subsheaf }F\hookrightarrow E \right\}.
\end{align*}
Then $\bigl(\Coh_L^{>s}(X),\Coh_L^{\leq s}(X)\bigr)$ is a torsion pair in $\Coh(X)$; see \cite{Happel-al:tilting}. Its tilted heart is
\[
\cA_s := \left\langle \Coh_L^{>s}(X), \Coh_L^{\leq s}(X)[1] \right\rangle.
\]

For $t>0$, define
\begin{equation*}
Z_{s,t}(E) = -\left( \ch_2(E)-sL\cdot\ch_1(E) +\frac{s^2-t^2}{2}L^2\ch_0(E) \right) +i\left( L\cdot\ch_1(E)-sL^2\ch_0(E) \right).
\end{equation*}
We write $\sigma_{s,t}:=(\cA_s,Z_{s,t})$ and $\nu_{s,t}$ for the slope $\mu_{Z_{s,t}}$.

\begin{proposition}\label{prop:stabsigma01}
The data $\sigma_0=(\cA_0,Z_{0,\sqrt2})$ is a stability condition on $\Db(X)$ satisfying the support property with respect to the quadratic form $\Delta$.
\end{proposition}

\begin{proof}
That $\sigma_0$ is a stability condition follows from the standard construction recalled above.  We verify that $\Delta$ gives its support property.

We first show that $\Delta$ is negative definite on $\Ker Z_{0,\sqrt2}$.  Let $0\neq v=(r,D,u)\in\Kn(X)\otimes\mathbb R$ satisfy $Z_{0,\sqrt2}(v)=0$.  Then
\[
u=rL^2, \qquad L\cdot D=0.
\]
By the Hodge index theorem, $D^2\leq0$, and hence
\[
\Delta(v) = D^2-2ru = D^2-2r^2L^2 \leq -2r^2L^2 \leq0.
\]
If equality holds, then $r=0$ and $D^2=0$. Since $L\cdot D=0$, the Hodge index theorem forces $D=0$, and then $u=0$.  Thus
\[
\Delta(v)<0 \qquad \text{for every } 0\neq v\in\Ker Z_{0,\sqrt2}.
\]

It remains to prove that $\Delta(E)\geq0$ for every $\sigma_0$-semistable object $E$. Suppose first that $E$ is $\sigma_0$-stable. An abelian surface has no nonzero rigid objects; see \cite[Lemma~15.1]{Bridgeland:K3}. Therefore, $\Ext^1(E,E)\neq0.$ Together with \eqref{eq:ext1even}, this gives $\dim\Ext^1(E,E)\geq2$, and consequently,
\[
\Delta(E) = -\chi(E,E) = \dim\Ext^1(E,E)-2 \geq0.
\]

Now let $E$ be strictly $\sigma_0$-semistable.  Its Jordan--H\"older factors are $\sigma_0$-stable and therefore have nonnegative discriminant. Since $\Delta$ is negative definite on $\Ker Z_{0,\sqrt2}$, \cite[Lemma~A.6]{BMS16} implies that $\Delta(E)\geq 0$. This proves the support property.
\end{proof}

\begin{remark}\label{rem:stabX}
More generally, $\Delta$ is negative definite on $\Ker Z_{s,t}$ for every $s\in\mathbb R$ and $t>0$.  Indeed, for any $v=(r,D,u)\in\Ker Z_{s,t}$, we have $L\cdot D=srL^2, u=(s^2+t^2)rL^2/2$. Writing $D=srL+D_\perp, L\cdot D_\perp=0$, we obtain $\Delta(v) = D_\perp^2-t^2r^2L^2$. By the Hodge index theorem, this is strictly negative unless $v=0$.

The family $\{\sigma_{s,t}\}_{s\in\mathbb R,\;t>0}$ is connected and contains $\sigma_0$.  It therefore follows from Proposition~\ref{propdef:stabQsigma} that the connected component $\Stab(X,\Delta,\sigma_0)$ contains the entire slice
\[
\Stab_L(X) = \{\sigma_{s,t}:s\in\mathbb R,\ t>0\}.
\]
Moreover, every stability condition in this component satisfies the support property with respect to the same quadratic form $\Delta$. In fact, we have $\Stab(X)=\Stab(X,\Delta,\sigma_0)$, see \cite{Dell25,FLZ:ab3,HMS:generic_K3s}. In the subsequent computation of actual walls, however, we shall only use the subspace $\Stab_L(X)$.
\end{remark}

\begin{remark}\label{rem:large-volume}
Let $E\in\Db(X)$ have positive rank and satisfy $\mu_L(E)>s$. The large-volume theorem \cite[Proposition~14.2]{Bridgeland:K3} states that, up to shift, $E$ is $\sigma_{s,t}$-semistable for all sufficiently large $t$ if and only if it is an $sL$-twisted Gieseker-semistable torsion-free sheaf. In particular, the corresponding sheaf is slope-semistable.  We shall use the large-volume limit both to identify the relevant wall factors, after the appropriate shifts, with torsion-free sheaves and to deduce their slope semistability.
\end{remark}

\begin{proposition}\label{prop:semihomogeneousobj}
Let $E\in\Db(X)$ be $\tau$-stable with $\Delta(E)=0$ for some $\tau\in\Stab(X,\Delta,\sigma_0)$. Then $E$ is $\sigma$-stable for every $\sigma\in\Stab(X,\Delta,\sigma_0)$. In particular, every simple semihomogeneous sheaf is stable throughout $\Stab_L(X)$.
\end{proposition}

\begin{proof}
Since $\tau$ belongs to $\Stab(X,\Delta,\sigma_0)$, the first assertion follows from Proposition~\ref{propdef:stabQsigma}(3), applied with $\tau$ as the base point.

Now let $E$ be a simple semihomogeneous sheaf. By \eqref{eq:2.1}, $\ch(E)=\rk{E}e^{\delta(E)}$, and hence $\Delta(E)=0$.

Moreover, $E$ is Gieseker-stable with respect to $L$ by \cite[Proposition~6.16]{Mukai78}.  Choosing $s<\mu_L(E)$, the large-volume theorem shows that $E$ is $\sigma_{s,t}$-stable for $t\gg0$.  The first assertion then implies that $E$ remains stable throughout $\Stab(X,\Delta,\sigma_0)$, and in particular throughout $\Stab_L(X)$.
\end{proof}

\section{The wall of \texorpdfstring{$\is q$}{Iq}}
\label{sec:first-wall}
\subsection{Stable factors of $\is q$}
We now determine the stable factors on an actual wall of the ideal sheaf of a point.  The use of the ordinary discriminant is essential.

\begin{proposition}\label{prop:twojhfactors}
Let $\sigma\in\Stab(X,\Delta,\sigma_0)$, and let $E\in\Db(X)$ satisfy $\Delta(E)=2$.  If $E$ is strictly $\sigma$-semistable, then it has exactly two $\sigma$-stable Jordan--H\"older factors $F$ and $Q$.  Moreover,
\[
\Delta(F)=\Delta(Q)=0, \qquad \langle\ch(F),\ch(Q)\rangle=-\chi(F,Q)=1,
\]
and
\[
\dim\Ext^1(F,Q)=\dim\Ext^1(Q,F)=1.
\]
\end{proposition}

\begin{proof} 
Let $E_1,\ldots,E_n$, with $n\geq2$, be the stable Jordan--H\"older factors.  Their numerical classes have the same phase. The Lorentzian linear algebra underlying the support property \cite[Lemma~A.6]{BMS16} shows that these classes lie in the same component of the nonnegative cone in the codimension-one subspace determined by that phase. Hence
\[
\langle\ch(E_i),\ch(E_j)\rangle\geq0 \qquad\text{for all }i,j.
\]
The diagonal terms are nonnegative even integers, and \[ 2=\Delta(E)  =\sum_i\Delta(E_i)   +2\sum_{i<j}\langle\ch(E_i),\ch(E_j)\rangle. \] 

It follows that every $E_i$ is isotropic and that the sum of the off-diagonal pairings is one. Thus precisely one pair has pairing one and every other off-diagonal pairing is zero.  In a Lorentzian space, two nonzero isotropic vectors in the same cone have pairing zero only when they are proportional.  A third factor would therefore be proportional to both members of the distinguished pair, which is impossible because their pairing is one. Hence $n=2$.

The two stable factors have the same phase and are nonisomorphic, so the $\Hom$ groups in both directions vanish.  Serre duality gives the same vanishing in degree two. Since $\chi(F,Q)=-1$, the two assertions about $\Ext^1$ follow.
\end{proof}

\begin{lemma}\label{lem:Delta0sheaf}
Let $E$ be a torsion-free $\mu_L$-semistable sheaf on $X$. If $\Delta(E)=0$, then $E$ is locally free.
\end{lemma}

\begin{proof} 
The reflexive hull $E^{\vee\vee}$ is again $\mu_L$-semistable: a destabilizing subsheaf would intersect $E$ in a subsheaf of the same rank and first Chern class and would therefore destabilize $E$.  If $T=E^{\vee\vee}/E$ had positive length, then
\[
\Delta(E^{\vee\vee})=\Delta(E)-2\rk{E}\operatorname{length}(T)<0,
\]
contrary to the Bogomolov inequality. Thus $T=0$.
\end{proof}

\begin{proposition}\label{prop:wallIq}
Suppose that $\is q$ is strictly $\sigma_{s,t}$-semistable at an actual wall in the half-plane $s<0$.  Order its stable factors so that there is a short exact sequence
\begin{equation}\label{eq:wallIq}
0\longrightarrow\shf F\longrightarrow\is q \longrightarrow\shf Q\longrightarrow0
\end{equation}
in $\cA_s$. Then $\shf F$ and $\shf Q$ are stable throughout $\Stab_L(X)$.  Writing
\[
\ch(\shf F)=(r,-C,u),
\]
one has
\begin{equation}\label{eq:wallIqch}
r\geq1,\qquad u=r-1.
\end{equation}
More precisely:
\begin{enumerate}
\item[(i)] If $r=1$, then
\begin{equation}\label{eq:wallIqrk1}
\shf F\simeq\ox(-C),\qquad 0\neq C\geq0,\qquad C^2=0, \qquad \shf Q\simeq\is q/\ox(-C)=:\mathscr I_{q,C}.
\end{equation}
\item[(ii)] If $r\geq2$, then $\shf Q\simeq\shf G[1]$, and
\eqref{eq:wallIq} is induced by a short exact sequence of coherent
sheaves
\begin{equation}\label{eq:wallIqshf}
0\longrightarrow\shf G\longrightarrow\shf F
\longrightarrow\is q\longrightarrow0.
\end{equation}
The sheaves $\shf F$ and $\shf G$ are simple semihomogeneous vector bundles with
\begin{equation}\label{eq:wallIqcharac}
\ch(\shf F)=(r,-C,r-1),
\qquad
\ch(\shf G)=(r-1,-C,r),
\end{equation}
and
\begin{equation}\label{eq:C2}
C^2=2r(r-1).
\end{equation}
In this case $L\cdot C>0$, and $C$ is ample.
\end{enumerate}
\end{proposition}

\begin{proof} 
Proposition~\ref{prop:twojhfactors}, applied to $\Delta(\is q)=2$, gives exactly two stable factors of discriminant zero. Their stability throughout the slice follows from Proposition~\ref{prop:semihomogeneousobj}.

Taking standard cohomology of \eqref{eq:wallIq} gives
\begin{equation}\label{eq:4.7}
0\longrightarrow\mathcal H^{-1}(\shf Q) \longrightarrow\shf F \longrightarrow\is q \longrightarrow\mathcal H^0(\shf Q) \longrightarrow0.
\end{equation}
Here $\mathcal H^{-1}(\shf F)=0$, so $\shf F$ is a coherent sheaf. Furthermore, $\mathcal H^{-1}(\shf Q)$ belongs to $\Coh_L^{\leq s}(X)$ and is torsion-free.  The map $\shf F\to\is q$ is nonzero, and its image is a rank-one torsion-free sheaf.  It follows from \eqref{eq:4.7} that $\shf F$ is torsion-free and that $r\geq1$.

Since
\[
\ch(\shf Q)=(1-r,C,-1-u),
\]
the two isotropic equalities give
\[
C^2=2ru=2(r-1)(u+1).
\]
Consequently $u=r-1$ and $C^2=2r(r-1)$.

Assume first that $r\geq2$ and put $E=\shf Q[-1]$. Then
\[
\ch(E)=(r-1,-C,r),\qquad\mathcal H^0(E)=\mathcal H^{-1}(\shf Q),\qquad\mathcal H^1(E)=\mathcal H^0(\shf Q).
\]
The object $E$ is stable throughout the slice. Choose a rational number
\[
s'<-\frac{L\cdot C}{(r-1)L^2}.
\]
The $s'L$-twisted degree of $E$ is positive. At large volume, \cite[Proposition~14.2]{Bridgeland:K3} therefore identifies $E$ with a shift of an $s'L$-twisted Gieseker-semistable torsion-free sheaf.  Since $\mathcal H^0(E)$ has positive rank, the shift is zero.  Hence $\mathcal H^1(E)=0$ and $E=\shf G$ is a torsion-free sheaf.  This proves \eqref{eq:wallIqshf} and \eqref{eq:wallIqcharac}.

Applying the large-volume statement also to $\shf F$ shows that $\shf F$ and $\shf G$ are $\mu_L$-semistable.  Lemma~\ref{lem:Delta0sheaf} then makes both sheaves locally free. They are simple because they are stable objects of the Bridgeland slice.  Since $\Delta=0$, Serre duality and Riemann--Roch give $\dim\Ext^1(\shf F,\shf F)=\dim\Ext^1(\shf G,\shf G)=2$. Mukai's theorem \cite[Theorem~5.8]{Mukai78} now implies that they are semihomogeneous.

The nonzero map $\shf F\to\is q$ and slope semistability give $\mu_L(\shf F)\leq0$, and hence $L\cdot C\geq0$.  Since $C^2=2r(r-1)>0$, the Hodge index theorem forces this inequality strict. Lemma~\ref{lem:positivecone} then shows that $C$ is ample.

Finally, suppose $r=1$.  A rank-one torsion-free sheaf has the form $M\otimes\mathcal I_W$.  The equality $\Delta(\shf F)=0$ forces $W=\emptyset$, so $\shf F=M$ is a line bundle.  The nonzero map $M\to\is q$ is injective and determines an effective divisor $C$ through $q$ with $M\simeq\ox(-C)$. The remaining assertions of \eqref{eq:wallIqrk1} follow from \eqref{eq:wallIqch}.
\end{proof}

\subsection{Cohomological rank function}
We next connect this wall description to the cohomological rank function. For an object $E$, write
\[
P_E(s)=\ch_2^s(E) =\ch_2(E)-sL\cdot\ch_1(E)+ds^2\ch_0(E),\qquad s\in \R.
\]
We shall use the notation for the zero-slope loci introduced by Lahoz--Rojas \cite{LR23}.   When $P_E(s)$ has two zeros, we write them as $p_E^-\leq p_E^+$; when it has only one zero, we simply write the zero as $p_E$.

\begin{proposition}[The first wall and the basepoint-freeness threshold]
\label{prop:1stwall}
If $\is q$ has no actual wall in the region $s<0$, then
\begin{equation}\label{eq:4.8}
\epsilon_1(l)=\tfrac1{\sqrt d}.
\end{equation}
Suppose now that its actual wall is described by Proposition~\ref{prop:wallIq}.
\begin{enumerate}
\item[(i)] If $r=1$ and $m=L\cdot C$, then
\[
p_{\shf Q}=-\frac1m <-\frac1{\sqrt d} <-\frac md=p_{\shf F}^{-}<0,
\]
and
\begin{equation}\label{eq:4.9}
\epsilon_1(l)=-p_{\shf Q}=\frac1{L\cdot C}.
\end{equation}
\item[(ii)] If $r\geq2$, put $L$-discriminant $\overline\Delta=(L\cdot C)^2-4dr(r-1)$, then
\begin{align}
p_{\shf F}^{\pm}
=\frac{-L\cdot C\pm\sqrt{\overline\Delta}}{2rd},\qquad
 p_{\shf Q}^{\pm}
=\frac{-L\cdot C\pm\sqrt{\overline\Delta}}{2(r-1)d},
\label{eq:wallroots}
\end{align}
where the roots for $\shf Q=\shf G[1]$ are those of $P_{\shf G}$. The relevant roots satisfy
\begin{equation}\label{eq:rootordering}
p_{\shf Q}^{+}<-\frac1{\sqrt d}<p_{\shf F}^{-},
\end{equation}
and
\begin{equation}\label{eq:rootthreshold}
\epsilon_1(l)=-p_{\shf Q}^{+} =\frac{L\cdot C-\sqrt{\overline\Delta}}{2(r-1)d} =\frac{2r}{L\cdot C+\sqrt{\overline\Delta}}.
\end{equation}
\end{enumerate}
\end{proposition}

\begin{proof}
By \cite[Definition~3.1 and Theorem~D]{LR23}, the HN filtration computes the cohomological rank function:
if $E_j$ are the $\nu_{s,0}$-HN factors of $E\in\cA_s$, then
\begin{equation}\label{eq:4.13}
h^1_{E,l}(-s) =\sum_{\nu_{s,0}(E_j)<0}-P_{E_j}(s).
\end{equation}
In particular,
\begin{equation*}
h^0_{E,l}(-s)-h^1_{E,l}(-s)=P_E(s).
\end{equation*}
Note that for $\is q$ one has
\[
P_{\is q}(s)=-1+ds^2, \qquad p_{\is q}^{-}=-\tfrac{1}{\sqrt d}.
\]
If $\is q$ has no actual wall for $s<0$, then it remains semistable and
\[
h^1_{\is q,l}(-s)=
\begin{cases}
1-ds^2,& p_{\is q}^{-}<s<0,\\
0,&s\leq p_{\is q}^{-}.
\end{cases}
\]
This proves \eqref{eq:4.8}.

Assume that an actual wall exists and put $m=L\cdot C$. Directly solving $\nu_{s,t}(\shf F)=\nu_{s,t}(\is q)$ gives the numerical wall
\begin{equation}\label{eq:first-wall-circle}
t^2+\left(s+\frac{2r-1}{m}\right)^2 =\frac{(2r-1)^2}{m^2}-\frac1d.
\end{equation}
Since the wall is actual, its radius is positive and therefore
\begin{equation}\label{eq:4.15}
m<(2r-1)\sqrt d.
\end{equation}

\noindent (i) If $r=1$, then
\[
P_{\shf F}(s)=s(m+ds), \qquad P_{\shf Q}(s)=-1-ms.
\]
The root ordering follows from \eqref{eq:4.15}. For $p_{\shf Q}<s\leq-1/\sqrt d$, by Proposition \ref{prop:wallIq}, the two stable factors form the weak HN filtration, with $P_{\shf F}(s)>0$ and $P_{\shf Q}(s)<0$. Formula \eqref{eq:4.13} gives
\[
h^1_{\is q,l}(-s)=-P_{\shf Q}(s)=1+ms>0.
\]
For $-1/\sqrt d<s<0$, the identity $h^0_{\is q,l}(-s)-h^1_{\is q,l}(-s)=P_{\is q}(s)<0$ gives the same positivity.  At $s=p_{\shf Q}$ both factors have nonnegative weak slope, so \eqref{eq:4.13} gives vanishing. This proves \eqref{eq:4.9}.

\noindent (ii) Now suppose $r\geq2$. The Hodge index theorem and
$C^2=2r(r-1)$ show that $\overline\Delta\geq0$, giving
\eqref{eq:wallroots}.  At $s=-1/\sqrt d$ one has
\[
P_{\shf F}(s)=P_{\shf G}(s)
=2r-1-\frac m{\sqrt d}>0
\]
by \eqref{eq:4.15}. Moreover,
\[
-\frac{m}{2(r-1)d}<-\frac1{\sqrt d}
<-\frac{m}{2rd},
\]
where the first inequality follows from the Hodge lower bound and the second from \eqref{eq:4.15}.  Locating $-1/\sqrt d$ relative to the vertices of the two quadratics gives \eqref{eq:rootordering}.

For $p_{\shf Q}^{+}<s\leq-1/\sqrt d$, the stable factors again give the weak HN filtration and
\[
h^1_{\is q,l}(-s)
=-P_{\shf Q}(s)=P_{\shf G}(s)>0.
\]
For $-1/\sqrt d<s<0$, positivity follows from
$P_{\is q}(s)<0$. At, or immediately to the left of,
$s=p_{\shf Q}^{+}$, both factors have nonnegative weak slope. More precisely, if $\overline\Delta>0$, the vanishing holds for rational
$s<p_{\shf Q}^{+}$ sufficiently close to the root; if
$\overline\Delta=0$, it holds at the rational root itself.  Hence
$\epsilon_1(l)=-p_{\shf Q}^{+}$.  The last expression in
\eqref{eq:rootthreshold} follows from $(L\cdot C)^2-\overline\Delta=4dr(r-1).$
\end{proof}

\subsection{The rank-one wall}

\begin{corollary}\label{cor:rk1true}
Assume that $X$ contains no elliptic curve $E$ with $L\cdot E\leq p+2$.  If the actual wall of $\is q$ has rank one, then
\[
\epsilon_1(l)<\frac1{p+2}.
\]
\end{corollary}

\begin{proof}
By Proposition~\ref{prop:wallIq}, the wall factor is $\shf F\simeq\ox(-C)$ for a nonzero effective divisor $C$ with $C^2=0$. Such a divisor is numerically a positive multiple of an elliptic curve. The hypothesis therefore gives $L\cdot C>p+2$.  Formula \eqref{eq:4.9} yields $\epsilon_1(l)=1/{L\cdot C}<1/{(p+2)}$.
\end{proof}

\section{Numerical consequences of the wall}
\label{sec:numerical-wall}
By Proposition ~\ref{prop:1stwall} and Corollary~\ref{cor:rk1true}, to study Property $N_p$ when $d>(p+2)^2$, we may focus on the case in which $\is q$ admits a higher rank destabilizing wall. Throughout this section, we assume that the actual wall of $\is q$ has rank $r\geq2$ and use the notation of Proposition~\ref{prop:wallIq}.
\subsection{Bound for $r$ and $L\cdot C$}
\begin{lemma}\label{lem:boundforr}
Assume $d>(p+2)^2$ and $\epsilon_1(l)\geq\frac1{p+2}$.  Then
\begin{equation}\label{eq:5.1}
r\leq\frac{d}{d-(p+2)^2}.
\end{equation}
Equality holds if and only if $\epsilon_1(l)=1/(p+2)$ and $\overline\Delta=0$.
\end{lemma}

\begin{proof}
First, suppose that $\epsilon_1(l)>1/(p+2)$.  By \eqref{eq:rootthreshold}, $p^+_{\shf Q}<-1/(p+2)$.  Evaluating $P_{\shf G}$ at $s=-1/(p+2)$ gives
\begin{equation}\label{bound 1}
L\cdot C<(p+2)r+\frac{(r-1)d}{p+2}.
\end{equation}

If $\epsilon_1(l)=1/(p+2)$, then $-\frac{1}{p+2}=p^+_{\shf Q}$, and $-\ch^s_2(\shf Q)>0$ for $s=-\frac{1}{p+2}+\epsilon$.  Consequently,
\begin{equation}\label{bound 2}
L\cdot C= (p+2)r+\frac{(r-1)d}{p+2}.
\end{equation}
In either case, we have the uniform estimate
\begin{equation}\label{eq:wallupperbound}
L\cdot C\leq(p+2)r+\frac{(r-1)d}{p+2}.
\end{equation}

The vertex of $P_{\shf G}$ lies to the left of its larger root. Hence
\[
\mu_L(\shf G)=-\frac{L\cdot C}{2(r-1)d}\leq p^+_{\shf Q}\leq-\frac1{p+2},
\]
and therefore
\[
L\cdot C\geq\frac{2(r-1)d}{p+2}.
\]
Combining this inequality with \eqref{eq:wallupperbound} gives \eqref{eq:5.1}.  Equality forces equality in both estimates, which is equivalent to $p^+_{\shf Q}=\mu_L(\shf G)=-1/(p+2)$; equivalently, $\epsilon_1(l)=1/(p+2)$ and $\overline\Delta=0$.
\end{proof}

The lemma gives the following slight strengthening of \cite[Theorem~1.2]{Ito18}.
\begin{corollary}\label{Ito18}
Suppose that $d\geq2(p+2)^2$ and that $X$ contains no elliptic curve $E$ with $L\cdot E\leq p+2$.  Then either $\epsilon_1(l)<1/(p+2)$, or
\begin{equation}\label{eq:5.5}
d=2(p+2)^2,\qquad
L\simeq M^{\otimes(p+2)},\qquad M^2=4.
\end{equation}
Here $M$ is ample.  In the exceptional case $L$ is of type $(p+2,2(p+2))$ and $\epsilon_1(l)=1/(p+2)$.
\end{corollary}

\begin{proof}
Assume that $\epsilon_1(l)\geq1/(p+2)$.  Formula \eqref{eq:4.8} excludes the no-wall case, and Corollary~\ref{cor:rk1true} excludes the rank-one wall. Thus the first wall has rank $r\geq2$.  Since
\[
\frac{d}{d-(p+2)^2}\leq2,
\]
Lemma~\ref{lem:boundforr} forces equality throughout:
\[
r=2,\qquad d=2(p+2)^2,\qquad\overline\Delta=0.
\]
Thus $C^2=4$, $L\cdot C=4(p+2)$, and $(L\cdot C)^2=L^2C^2$. The equality case of the Hodge index theorem gives the numerical equivalence $L\equiv(p+2)C$.

Put $M_0=\ox(C)$.  Then $\alpha=L\otimes M_0^{-\otimes(p+2)}$ belongs to $\Pic^0(X)$. The multiplication by $p+2$ on $\Pic^0(X)$ is surjective, so there is an $\eta\in\Pic^0(X)$ with $\eta^{\otimes(p+2)}\simeq\alpha$.  Setting $M=M_0\otimes\eta$ gives \eqref{eq:5.5}. Lemma~\ref{lem:boundforr} also gives $\epsilon_1(l)=1/(p+2)$.  Since $C$ is ample, so is $M$.

Conversely, if $L\simeq M^{\otimes(p+2)}$ with $M^2=4$, then $h^0(X,M)=2$, so $M$ is not basepoint-free.  Write $m$ for the numerical class of $M$.  Then $\epsilon_1(m)=1$ by Theorem~\ref{thm:caucci}, and homogeneity of the basepoint-freeness threshold gives $\epsilon_1(l)=\epsilon_1(m)/(p+2)$.
\end{proof}

\begin{lemma}\label{L.C}
Let $p\geq1$ and $d>(p+2)^2$.  Suppose that $\frac{1}{p+1}>\epsilon_1(l)\geq\frac{1}{p+2}$. Then
\[L\cdot C=(p+2)r+\lfloor\frac{(r-1)d}{p+2}\rfloor,\]
where $\lfloor-\rfloor$ denotes the integer part.
\end{lemma}

\begin{proof}
We distinguish two cases.
Put
\[
\beta_{\shf Q}:=\mu_L(\shf G) =-\frac{L\cdot C}{2(r-1)d},
\]
the vertex of the quadratic $P_{\shf G}$.

\emph{Case I:} First suppose that
$\beta_{\shf Q}\le-\frac{1}{p+1}<p^+_{\shf Q}$.  Then evaluating $P_{\shf G}$ at $-1/(p+1)$, and using \eqref{eq:wallupperbound}, gives
\[(p+1)r+\frac{(r-1)d}{p+1}<L\cdot C\le (p+2)r+\frac{(r-1)d}{p+2}.\]
Consider the difference
\begin{align*}
0
&<\left((p+2)r+\frac{(r-1)d}{p+2}\right)
 -\left((p+1)r+\frac{(r-1)d}{p+1}\right)\\
&=r-\frac{(r-1)d}{(p+2)(p+1)}\\
&=1-(r-1)\left(\frac{d}{(p+2)(p+1)}-1\right)<1.
\end{align*}
This forces
\[L\cdot C=(p+2)r+\lfloor\frac{(r-1)d}{p+2}\rfloor.\]
Moreover, the inequality
\[r-\frac{(r-1)d}{(p+2)(p+1)}>0\]
leads to a slightly better bound for $r$:
\[r<\frac{d}{d-(p+2)(p+1)}.\]
Together with $r\geq2$, this gives the sharper bounds
\[r\le p+1 \quad \text{and} \quad d<2(p+2)(p+1).\]

\emph{Case II:}
Now suppose that
$-\frac{1}{p+1}<\beta_{\shf Q}=-\frac{L\cdot C}{2(r-1)d}$, or equivalently $L\cdot C<\frac{2(r-1)d}{p+1}$.  The Hodge index theorem gives
\[4r(r-1)d=L^2C^2\le (L\cdot C)^2< \frac{2(r-1)d}{p+1}\paren{(p+2)r+\frac{(r-1)d}{p+2}},\]
yielding the first inequality
\begin{equation}\label{eq:5.6}
rp(p+2)<(r-1)d\le r(p+2)^2,
\end{equation}
where the second inequality follows from Lemma \ref{lem:boundforr}. 

Performing division with remainder, write
\[
(r-1)d=(p+2)\alpha+\delta,
\]
where integers $\alpha=
\left\lfloor\frac{(r-1)d}{p+2}\right\rfloor
\text{ and }
0\le\delta\le p+1.$
Equation~\eqref{eq:5.6} implies that
\[
rp\le\alpha\le r(p+2).
\]
Set
\[
e=(p+2)r-\alpha.
\]
Then the above inequalities amount to
\begin{equation}\label{eq:ebounds}
0\le e\le 2r.
\end{equation}
Moreover, since $d>(p+2)^2$, we have
\[
(r-1)d>(r-1)(p+2)^2.
\]
Using
\[
(r-1)d
=(p+2)\alpha+\delta
=r(p+2)^2-(p+2)e+\delta,
\]
we obtain
\[
(p+2)e<(p+2)^2+\delta
<(p+2)^2+(p+2).
\]
Hence
\[
e\le p+2.
\]
Combining this with \eqref{eq:ebounds}, we deduce that
\[
e^2\le 2r(p+2).
\]
Furthermore, since $r\ge2$,
\[
2e+1
\le 2(p+2)+1
<4(p+2)
\le 2r(p+2).
\]
It follows that
\[
(e+1)^2<4r(p+2).
\]

We now compute
\[
\begin{aligned}
4r\bigl((p+2)\alpha+\delta\bigr)
 -\bigl(r(p+2)+\alpha-1\bigr)^2 =4r(p+2+\delta)-(e+1)^2
>0.
\end{aligned}
\]
Thus
\[
\bigl(r(p+2)+\alpha-1\bigr)^2
<
4r\bigl((p+2)\alpha+\delta\bigr)
=L^2C^2
\le (L\cdot C)^2.
\]
Therefore,
\[
L\cdot C>r(p+2)+\alpha-1.
\]
On the other hand, \eqref{eq:wallupperbound} and the integrality of
$L\cdot C$ give
\[
L\cdot C\le (p+2)r+\alpha.
\]
We conclude that
\[
L\cdot C
=(p+2)r+\alpha
=(p+2)r+
\left\lfloor\frac{(r-1)d}{p+2}\right\rfloor.
\]
\end{proof}

\begin{remark}
Lemma \ref{L.C} also provides a heuristic explanation for why $N_p$ holds for a general polarized abelian surface when $d>(p+2)^2$: for given $(X, L)$, the numbers $d, r$ and the curve class $C$ are determined accordingly, it is rare that the Diophantine equation $L\cdot C=rx+\lfloor\frac{(r-1)d}{x}\rfloor$ admits a solution, although it is possible. 
\end{remark}

\subsection{Proof of Theorem~\ref{main thm3}}
\begin{proof}[Proof of Theorem~\ref{main thm3}]
Since
\[
2(p+2)^2-2(p+2)+1>2(p+1)^2,
\]
Corollary~\ref{Ito18}, applied with $p$ replaced by $p-1$, gives
\[
\epsilon_1(l)<\frac{1}{p+1}.
\]

If $d\geq 2(p+2)^2$, Corollary~\ref{Ito18}, now applied with $p$, gives the desired conclusion.  Its exceptional case is excluded by the primitivity of $L$.  Thus it remains to consider
\[
2(p+2)^2-2(p+2)+1\leq d\leq 2(p+2)^2-1.
\]
Suppose to the contrary that $\epsilon_1(l)\geq\frac{1}{p+2}$. The no-wall value in \eqref{eq:4.8} is smaller than $1/(p+2)$, and Corollary~\ref{cor:rk1true} excludes a rank-one wall. Thus the first wall has rank $r\geq2$, and Lemma~\ref{L.C} gives
\[
L\cdot C=(p+2)r+
\left\lfloor\frac{(r-1)d}{p+2}\right\rfloor,
\]
where
\[
C^2=2r(r-1).
\]
Moreover, Lemma \ref{lem:boundforr} gives
\[
r\leq\frac{d}{d-(p+2)^2}.
\]

Assume first that $p\geq2$. The lower bound on $d$ yields
\[
2d\geq4(p+2)^2-4(p+2)+2>3(p+2)^2.
\]
Consequently, $\frac{d}{d-(p+2)^2}<3$, and hence $r=2$. Lemma \ref{L.C} now becomes
\[
L\cdot C =2(p+2)+\left\lfloor\frac{d}{p+2}\right\rfloor.
\]

We divide the argument into three cases. 

Case I: Assume 
\[
2(p+2)^2-2(p+2)+1 \leq d\leq 2(p+2)^2-(p+2)-1.
\]
Then
\[
L\cdot C=4(p+2)-2.
\]
Since $C^2=4$, the Hodge index theorem gives
\begin{eqnarray*}
\bigl(4(p+2)-2\bigr)^2
&=&(L\cdot C)^2 \geq L^2C^2=8d \\
&\geq & 8\bigl(2(p+2)^2-2(p+2)+1\bigr)\\ &=&\bigl(4(p+2)-2\bigr)^2+4,
\end{eqnarray*}
a contradiction.

Case II: Assume
\[
2(p+2)^2-(p+2)+1 \leq d\leq 2(p+2)^2-1.
\]
Then
\[
L\cdot C=4(p+2)-1.
\]
The Hodge index theorem similarly gives
\[
\begin{aligned}
\bigl(4(p+2)-1\bigr)^2
&=(L\cdot C)^2\geq 8d \geq \bigl(4(p+2)-1\bigr)^2+7,
\end{aligned}
\]
again a contradiction.

Case III: The only remaining degree is $d=2(p+2)^2-(p+2).$
In this case, Lemma~\ref{L.C} gives
\[
L\cdot C=4(p+2)-1.
\]
Set
\[
E=(p+2)C-L.
\]
Using $C^2=4$ and $L^2=2d$, we compute
\[
\begin{aligned}
E^2
&=(p+2)^2C^2-2(p+2)L\cdot C+L^2
=0,\\
L\cdot E
&=(p+2)L\cdot C-L^2
=p+2.
\end{aligned}
\]
On an abelian surface, a nonzero integral class of square zero having positive intersection with an ample class is numerically equivalent to a positive multiple of an elliptic curve. Thus there exist an elliptic curve $E_0\subseteq X$ and an integer $m\geq1$ such that
\[
E\equiv mE_0.
\]
It follows that
\[
L\cdot E_0=\frac{p+2}{m}\leq p+2,
\]
contrary to the hypothesis.

It remains to treat $p=1$. The relevant degrees are $13\leq d\leq17$. By Lemma \ref{lem:boundforr}, $r=2$, except that $r=3$ is also possible when $d=13$.

Suppose first that $r=2$. Lemma~\ref{L.C} gives
\[L\cdot C=6+\lfloor\tfrac{d}{3}\rfloor=\left\{\begin{array}{ll}
    10 & \textrm{if $d=13, 14$,}\\
    11 & \textrm{if $d=15, 16, 17$.}
   \end{array}\right.\]

If $d\neq 15$, then $(L\cdot C)^2<L^2C^2=8d$, contradicting the Hodge index theorem. 

If $d=15$, then the class $E=3C-L$ satisfies $E^2=0$ and $L\cdot E=3$, contrary to the hypothesis.

Finally, if $(d,r)=(13,3)$, Lemma \ref{lem:boundforr} gives
\[
L\cdot C
=9+\lfloor\tfrac{2d}{3}\rfloor
=17.
\]
Then $(L\cdot C)^2=17^2<L^2C^2=312$, contradicting the Hodge index theorem again.

Therefore $\epsilon_1(l)<\frac{1}{p+2}$, as claimed.
\end{proof}

\begin{proposition}\label{p=0}
Suppose that $d\geq7$ and that $X$ contains no elliptic curve $E$ with $L\cdot E\leq2$.  Then $\epsilon_1(l)\leq 1/2$. Equality holds if and only if $L\simeq M^{\otimes2}$ for an ample line bundle $M$ with $M^2=4$.  In this case $|M|$ has no fixed component.
\end{proposition}

\begin{proof}
Suppose first that $d=7$ and $\epsilon_1(l)\geq1/2$. The no-wall value $1/\sqrt7$ is smaller than $1/2$, and the rank-one case is excluded by Corollary~\ref{cor:rk1true}.  Thus $r\geq2$, and Lemma~\ref{lem:boundforr} gives
\[
2\leq r\leq\tfrac{7}{7-4}=\tfrac73.
\]
Hence $r=2$.  If $\epsilon_1(l)>1/2$, then \eqref{bound 1} gives
\[
L\cdot C<2r+\tfrac{(r-1)d}{2}=\tfrac{15}{2},
\]
and therefore $L\cdot C\leq7$.  This contradicts the Hodge index theorem,
since
\[
(L\cdot C)^2\geq L^2C^2=14\cdot4=56.
\]
If $\epsilon_1(l)=1/2$, then \eqref{bound 2} gives
$L\cdot C=15/2$, which is impossible.  Thus
$\epsilon_1(l)<1/2$ when $d=7$.

Now suppose that $d\geq8$.  Corollary~\ref{Ito18}, applied with $p=0$,
gives $\epsilon_1(l)<1/2$ unless $d=8$ and
\[
L\simeq M^{\otimes2},\qquad M^2=4.
\]
Conversely, suppose that $L\simeq M^{\otimes2}$ with $M^2=4$.  Then $h^0(X,M)=2$, so $|M|$ has a base point. Let $m$ denote the numerical class of $M$. By Theorem~\ref{thm:caucci}, $\epsilon_1(m)=1$. It follows that
\[
\epsilon_1(l)=\epsilon_1(2m)=\tfrac{\epsilon_1(m)}2=\tfrac12.
\]
Finally, the hypothesis on elliptic curves and Reider's very-ampleness criterion imply that $L=M^{\otimes2}$ is very ample.  By \cite[Lemma~2]{Ohbuchi93}, this is equivalent to $|M|$ having no fixed component.
\end{proof}

We end the section with a remark. 
\begin{remark}
The inequality $\epsilon_1(l)<\frac{1}{p+2}$ need not hold under the hypotheses of Theorem~\ref{main thm}.  For example, a very general polarized abelian surface of type $(1,12)$ satisfies $\epsilon_1(l)=\frac{1}{3}$, while very general polarized abelian surfaces of types $(1,17)$, $(1,18)$, and $(1,20)$ satisfy $\epsilon_1(l)=\frac{1}{4}$; see \cite[Remark 4.2(2)]{Rojas22} and \cite[Theorem 1.3]{Ito26}. These surfaces contain no elliptic curves. Nevertheless, by Theorem \ref{main thm}, the polarization of type $(1,12)$ satisfies Property $N_1$, while those of types $(1,17)$, $(1,18)$, and $(1,20)$ satisfy Property $N_2$. Thus, the criterion of Jiang--Pareschi and Caucci is sufficient but not necessary. These examples illustrate the need for the Fourier--Mukai argument developed below.
\end{remark}

\section{A Fourier--Mukai presentation of \texorpdfstring{$M_L$}{ML} and cohomology vanishings}
\label{sec:mlcompute}

For the remainder of the proof, fix $p\geq1$ and assume
\begin{equation}\label{eq:stamainrange}
d\geq(p+2)^2+1,\qquad
\epsilon_1(l)\geq\frac1{p+2},\qquad
L\cdot E\geq p+3
\end{equation}
for every elliptic curve $E\subseteq X$.  The no-wall case is excluded by \eqref{eq:4.8}, and Corollary~\ref{cor:rk1true} excludes the rank-one wall.  We therefore use the higher-rank sequence \eqref{eq:wallIqshf}. Since $\epsilon_1(l)$ is independent of the point, we take $q=0$ from now on.  Reider's theorem implies that $L$ is very ample. 

Put
\begin{equation}\label{eq:abU}
a=(r-1)d-L\cdot C+r, \qquad b=rd-L\cdot C+r-1, \qquad U=\frac{(r-1)d}{p+2}-r.
\end{equation}
Since $d>(p+2)^2$, we always have
\[
U>(r-1)(p+2)-r\geq p>0.
\]
\subsection{A Fourier--Mukai presentation of \texorpdfstring{$M_L$}{ML}}
\label{sec:FM}

\begin{lemma}\label{lem:twisted-wall-IT0}
The numbers $a$ and $b$ are positive, and the semihomogeneous bundles $L\otimes\shf G$ and $L\otimes\shf F$ are $\IT(0)$.
\end{lemma}

\begin{proof}

The upper bound \eqref{eq:wallupperbound} gives
\begin{equation}\label{eq:apositive}
a\geq(p+1)U>0,
\qquad
b=a+d-1>0.
\end{equation}

The normalized first Chern classes of $L\otimes\shf G$ and $L\otimes\shf F$ are, respectively,
\[
L-\frac{C}{r-1},
\qquad
L-\frac Cr.
\]
Using \eqref{eq:C2}, their squares are
\begin{equation*}
\left(L-\frac{C}{r-1}\right)^2=\frac{2a}{r-1}>0,
\qquad
\left(L-\frac Cr\right)^2=\frac{2b}{r}>0.
\end{equation*}
Their intersections with $L$ are positive. Indeed,
\begin{align*}
2d-\frac{L\cdot C}{r-1}
&\geq d\left(2-\frac{1}{p+2}\right)
-\frac{(p+2)r}{r-1}\\
&>(p+2)^2\left(2-\frac{1}{p+2}\right)-2(p+2)>0,
\end{align*}
and
\begin{align*}
2d-\frac{L\cdot C}{r}
&\geq2d-(p+2)-\frac{(r-1)d}{(p+2)r}\\
&>d\left(2-\frac{1}{p+2}\right)-(p+2)>0.
\end{align*}
By Lemma~\ref{lem:positivecone}, both normalized classes are ample. The assertion follows from Lemma~\ref{lem:index}(i).
\end{proof}

Define the skew Pontryagin products
\begin{equation*}
\mathcal A:=L\widehatstar\shf G,
\qquad
\mathcal B:=L\widehatstar\shf F.
\end{equation*}

\begin{proposition}\label{prop:FM}
There is a short exact sequence of vector bundles
\begin{equation}\label{eq:ABM}
0\longrightarrow\mathcal A\longrightarrow\mathcal B \longrightarrow M_L\longrightarrow0.
\end{equation}
Moreover, $\mathcal A$ and $\mathcal B$ are semihomogeneous with
\begin{equation}\label{eq:shfABdata}
\begin{aligned}
\rk{\mathcal A}&=a,
&\delta(\mathcal A)&=\frac{rL-dC}{a},\\
\rk{\mathcal B}&=b,
&\delta(\mathcal B)&=\frac{(r-1)L-dC}{b}.
\end{aligned}
\end{equation}
\end{proposition}

\begin{proof}
Tensor \eqref{eq:wallIqshf} by $L$ and apply the Fourier--Mukai transform.  Lemma~\ref{lem:twisted-wall-IT0}, together with the fact that $L\otimes\is 0$ is $\IT(0)$, shows that the transformed triangle is a short exact sequence of locally free sheaves.  Here $L\otimes\is 0$ is $\IT(0)$ because $L$ is globally generated and every twist $L\otimes\alpha$ is a translate of $L$.  Pulling back by $\varphi_l$ and tensoring by $L$ gives
\[
0\longrightarrow\mathcal A\longrightarrow\mathcal B
\longrightarrow L\widehatstar\is 0\longrightarrow0.
\]
Equation \eqref{eq:skewkernel} identifies the last term with $M_L$.

The discussion preceding Lemma~\ref{lem:FM-calculation} shows that $\mathcal A$ and $\mathcal B$ are semihomogeneous. Applying Lemma~\ref{lem:FM-calculation} to the Chern characters of $\shf{F}$ and $\shf{G}$ in \eqref{eq:wallIqcharac} gives \eqref{eq:shfABdata}. 
\end{proof}

\subsection{The numerical estimate}\label{sec:numerical}

The following lemma is the numerical core of the proof.  
\begin{lemma}\label{lem:numerical}
Let $p\geq1$, $d\geq(p+2)^2+1$, and $r\geq2$ be integers.  Suppose that $L\cdot C$ is a positive integer satisfying
\begin{equation}\label{eq:arith-hodge}
(L\cdot C)^2\geq4dr(r-1)
\quad \text{ and }\quad 
L\cdot C\leq(p+3)r+U.
\end{equation}
Define $a,b$ by \eqref{eq:abU} and put
\begin{equation}\label{eq:Theta}
\Theta(p,d,r,L\cdot C)
=2+\frac{2r-L\cdot C}{a}
+p\frac{2r-2-L\cdot C}{b}.
\end{equation}
Then $\Theta>0$, except possibly for
\begin{equation}\label{eq:arith-exceptions}
(p,d,r,L\cdot C)=(1,10,2,9),\quad(1,12,2,10),\quad(2,17,2,12).
\end{equation}
\end{lemma}

\begin{proof}
Step 1: We reformulate $\Theta$ and drop the term $L\cdot C$. 

As in \eqref{eq:apositive}, $U>p>0$, $a\geq(p+1)U>0$, and $b=a+d-1>0$.

For fixed $p,d,r$, both nonconstant summands in \eqref{eq:Theta} are strictly decreasing functions of $L\cdot C$.  Indeed,
\begin{align*}
\frac{\partial}{\partial(L\cdot C)}
\frac{2r-L\cdot C}{(r-1)d-L\cdot C+r}
&=\frac{r-(r-1)d}{a^2}<0,\\
\frac{\partial}{\partial(L\cdot C)}
\frac{2r-2-L\cdot C}{rd-L\cdot C+r-1}
&=\frac{r-1-rd}{b^2}<0.
\end{align*}
It is therefore enough to substitute the upper bound $L\cdot C=(p+3)r+U$.

Direct substitution gives
\begin{equation}\label{eq:Theta-boundary}
\Theta
=2-\frac{1}{p+1}-\frac rU
-\frac{p(r-1)(U+r(p+1)+2)}{(r(p+1)+1)(U+1)}.
\end{equation}
The derivative of the right-hand side with respect to $U$ is
\begin{equation}\label{eq:Theta-increasing}
\frac r{U^2}+\frac{p(r-1)}{(U+1)^2}>0.
\end{equation}
Thus the boundary value is strictly increasing in $d$.

Step 2: We take $d=(p+2)^2+4$ and show that there is no outlier case. Put
\[
\rho=r-1\geq1,
\qquad
t=p-1\geq0.
\]
After multiplying \eqref{eq:Theta-boundary} by the positive denominator
\[
(p+2)^2(p+1)U(r(p+1)+1)(U+1),
\]
the numerator is $\rho P_4(\rho,t)$, where
\begin{align*}
  P_4(\rho,t)&=A\rho^2+B\rho+D,\\
A&=4(t+2)^2(t^2+5t+10)>0,\\
B&=(t+3)(t^5+10t^4+49t^3+128t^2+188t+112)>0,\\
    D&=-2p^5-16p^4-62p^3-132p^2-136p-48.
\end{align*}
It follows that
\begin{align*}
P_4(1,t) &=t^6+11t^5+57t^4+165t^3+274t^2+256t+100>0\\
 \implies     P_4(\rho,t)&=P_4(1,t)+A(\rho^2-1)+B(\rho-1)>0.
\end{align*}
By \eqref{eq:Theta-increasing}, this proves $\Theta>0$ for every
$d\geq(p+2)^2+4$.

Step 3: It remains to consider $d=(p+2)^2+\delta$ with $\delta\in\{1,2,3\}$.  With the same positive denominator, the numerator is $\rho P_\delta(\rho,t)$, where
\[
P_\delta(\rho,t)=A_\delta(t)\rho^2+B_\delta(t)\rho+D_\delta(t).
\]
The coefficients needed below are
\begin{align*}
A_1(t)&=t^4+9t^3+31t^2+48t+28,\\
B_1(t)&=t^6+13t^5+70t^4+197t^3+299t^2+223t+57,\\
P_1(1,t)&=t^6+11t^5+45t^4+69t^3-35t^2-218t-176;
\end{align*}
\begin{align*}
A_2(t)&=2t^4+18t^3+64t^2+104t+64,\\
B_2(t)&=t^6+13t^5+73t^4+223t^3+386t^2+356t+132,\\
P_2(1,t)&=t^6+11t^5+49t^4+101t^3+62t^2-86t-110;
\end{align*}
\begin{align*}
A_3(t)&=3t^4+27t^3+99t^2+168t+108,\\
B_3(t)&=t^6+13t^5+76t^4+249t^3+477t^2+507t+225,\\
P_3(1,t)&=t^6+11t^5+53t^4+133t^3+165t^2+72t-18.
\end{align*}
Every $A_\delta(t)$ and $B_\delta(t)$ is positive for $t\geq0$, so $P_\delta(\rho,t)$ is strictly increasing in $\rho\geq1$. 
It remains to determine the sign of \(P_\delta(1,t)\). For every
integer \(u\geq0\), we have
\begin{align*}
P_1(1,u+2)
&=u^6+23u^5+215u^4+1029u^3+2579u^2+2982u+936,\\
P_2(1,u+1)
&=u^6+17u^5+119u^4+427u^3+784u^2+598u+28,\\
P_3(1,u+1)
&=u^6+17u^5+123u^4+475u^3+1007u^2+1074u+417,
\end{align*}
which are all positive. We are reduced to consider the following finite list of cases:
\[
\begin{array}{c|c|c|c}
\delta&(t,\rho)&P_\delta(\rho,t)&\text{first positive value in }\rho\\ \hline
1&(0,1)&-176&P_1(3,0)=162\\
1&(0,2)&-35&\\
1&(1,1)&-303&P_1(2,1)=908\\
2&(0,1)&-110&P_2(2,0)=214\\
3&(0,1)&-18&P_3(2,0)=531.
\end{array}
\]
Therefore \eqref{eq:Theta-boundary} can be non-positive only when
\[
(p,d,r)\in
\{(1,10,2),(1,10,3),(2,17,2),(1,11,2),(1,12,2)\}.
\]

Step 4: We combine the integrality of $L\cdot C$ with \eqref{eq:arith-hodge} to determine the remaining cases. For $(p,d,r)=(1,10,2)$ one obtains $L\cdot C=9$.  For $(1,10,3)$ the upper bound gives $L\cdot C\leq15$, whereas $(L\cdot C)^2\geq240$, which is impossible.  For $(2,17,2)$ one obtains $L\cdot C=12$.  For $(1,11,2)$ one has $L\cdot C\leq9$ but $(L\cdot C)^2\geq88$, which is impossible. Finally, for $(1,12,2)$ one obtains $L\cdot C=10$. This proves the lemma.
\end{proof}

\subsection{The two cohomology vanishings}\label{sec:cohomology}

We keep the notation of Proposition~\ref{prop:FM}.

\subsubsection{The \texorpdfstring{$\mathcal B$}{B}-term}

\begin{proposition}\label{positivity of B term}
The $\mathbb Q$-twisted bundle $\mathcal B\left\langle\frac1{p+1}l\right\rangle$ is $\IT(0)$.
\end{proposition}

\begin{proof}
The bundle $\shf{F}^{\vee}$ is simple semihomogeneous. Hence, by \cite[Theorem 5.8]{Mukai78} there exists an isogeny $\varphi:X'\to X$ and a line bundle $\mathcal O_{X'}(M)$ on $X'$ such that
\[
\varphi_*\mathcal O_{X'}(M)\simeq\shf F^\vee.
\]
Comparing Chern characters and using the projection formula gives
\begin{equation}\label{eq:Mukaidata}
\deg\varphi=r,\qquad
M^2=2(r-1),\qquad
\varphi^*L\cdot M=L\cdot C.
\end{equation}
In particular, $M^2>0$ and $M\cdot\varphi^*L>0$, so Lemma~\ref{lem:positivecone}, applied on $X'$, shows that $M$ is ample.

By the identity \eqref{eq:skewFM} and \cite[Proposition 4.1]{Ito22}, we deduce that
\begin{align*}
\mathcal B\left\langle\frac1{p+1}l\right\rangle
\text{ is $\IT(0)$}
&\Longleftrightarrow
\shf F\left\langle\frac1{p+2}l\right\rangle
\text{ is $\IT(0)$}\\
&\Longleftrightarrow
\mathcal O_{X'}(-M)\left\langle\frac1{p+2}\varphi^*l\right\rangle
\text{ is $\IT(0)$}\\
&\Longleftrightarrow
D:=\varphi^*L-(p+2)M\text{ is ample}.
\end{align*}
On the other hand, using \eqref{eq:wallupperbound} and \eqref{eq:Mukaidata}, we
have that
\begin{align*}
D^2
&=2rd-2(p+2)L\cdot C+2(p+2)^2(r-1)\\
&\geq2\bigl(d-(p+2)^2\bigr)>0.
\end{align*}
The Hodge lower bound gives
\[
L\cdot C\geq2\sqrt{dr(r-1)}
>2(p+2)(r-1),
\]
where the strict inequality uses the assumption $d>(p+2)^2$.  Consequently,
\[
D\cdot M=L\cdot C-2(p+2)(r-1)>0.
\]
Since $M$ is ample, Lemma~\ref{lem:positivecone} shows that $D$ is ample. This completes the proof.
\end{proof}
\begin{remark}
The analogous assertion unfortunately fails for $\mathcal A=L\widehatstar\shf G$. Indeed, let $\varphi$ be an isogeny such that $\varphi_*\sshf{X'}(N)\iso\shf G^{\vee}$. Then $\paren{\varphi^*L-(p+2)N}\cdot N=L\cdot C-2(p+2)r\leq0$ by \eqref{eq:wallupperbound} and \eqref{eq:5.1}, so $\varphi^*L-(p+2)N$ is not ample.
\end{remark}

\begin{corollary}\label{prop:B-vanishing}
For $1\leq k\leq p+1$ and $h\geq1$, we have
\[
H^1(X,\bigwedge^{k}\mathcal B\otimes L^{\otimes h})=0.
\]
\end{corollary}
\begin{proof}
We have
\[
{\mathcal B}^{\otimes k}\otimes L^{\otimes h}
=\mathcal B\left\langle\frac1{p+1}l\right\rangle^{\!\otimes k}
\otimes\ox\left\langle\frac{h(p+1)-k}{p+1}l\right\rangle.
\]
By Proposition~\ref{positivity of B term}, $\mathcal B\langle\frac{1}{p+1}l\rangle$ is $\IT(0)$.  Since $h(p+1)-k\geq0$, the second factor is GV.  Repeatedly applying the tensor-product theorem for $\mathbb Q$-twisted GV sheaves \cite[Proposition~3.4]{Caucci20} shows that $\mathcal B^{\otimes k}\otimes L^{\otimes h}$ is $\IT(0)$, and hence $H^1(X,\mathcal B^{\otimes k}\otimes L^{\otimes h})=0$.  The desired vanishing follows because $\bigwedge^k\mathcal B\otimes L^{\otimes h}$ is a direct summand of $\mathcal B^{\otimes k}\otimes L^{\otimes h}$.
\end{proof}

\subsubsection{The \texorpdfstring{$\mathcal A$}{A}-term}

For integers $1\leq k\leq p+1$ and $h\geq1$, put
\begin{equation*}
\mathcal T_{h,k}
=\mathcal A\otimes\bigwedge^{k-1}\mathcal B\otimes L^{\otimes h}.
\end{equation*}
Since $b=a+d-1\geq d>p$, the exterior power above
is nonzero throughout the required range.

\begin{lemma}\label{lem:Tsemihomogeneous}
The bundle $\mathcal T_{h,k}$ is semihomogeneous, and
\begin{equation}\label{eq:Dhk}
\delta(\mathcal T_{h,k})
=hL+\frac{rL-dC}{a}
+(k-1)\frac{(r-1)L-dC}{b}.
\end{equation}
\end{lemma}

\begin{proof}
Since $\mathcal A$, $\mathcal B$ and $L$ are all semihomogeneous, the properties presented in \S~\ref{subsec: semihomogeneous} show that $\mathcal T_{h,k}$ is semihomogeneous. 
Formula \eqref{eq:Dhk} follows from \eqref{eq:shfABdata} and \eqref{eq:2.2}.
\end{proof}

\begin{proposition}\label{prop:T-vanishing}
For every $1\leq k\leq p+1$ and $h\geq1$,
\[
H^2(X,\mathcal T_{h,k})=0.
\]
\end{proposition}

\begin{proof}
Intersecting \eqref{eq:Dhk} with $L$ gives
\begin{equation}\label{eq:Dhk-degree}
L\cdot\delta(\mathcal T_{h,k})
=d\left(
2h+\frac{2r-L\cdot C}{a}
+(k-1)\frac{2r-2-L\cdot C}{b}
\right).
\end{equation}
By \eqref{eq:C2} and the Hodge index theorem, $L\cdot C\geq \sqrt{2dC^2}>2r$. Hence both fractions in \eqref{eq:Dhk-degree} are negative. It follows that the smallest value of $L\cdot\delta(\mathcal T_{h,k})$ in the range of $h$ and $k$ occurs for $h=1, k=p+1$. It is equal to
\[
d\,\Theta(p,d,r,L\cdot C).
\]

Outside the three tuples in \eqref{eq:arith-exceptions}, Lemma~\ref{lem:numerical} shows that this number is positive.  Since $\mathcal T_{h,k}$ is semihomogeneous, it is slope-semistable.  Its dual is semistable of negative slope, so
\[
H^0(X,\mathcal T_{h,k}^{\vee})=0.
\]
By Serre duality, 
\[
H^2(X,\mathcal T_{h,k})=0,
\]

It remains to treat the three tuples in \eqref{eq:arith-exceptions}.  Direct substitution in \eqref{eq:Dhk} gives
\[
\begin{array}{c|c|c|c|c}
(p,d,r,L\cdot C)&a&b&L\cdot\delta(\mathcal T_{1,p+1})
&\delta(\mathcal T_{1,p+1})^2\\ \hline
(1,10,2,9)&3&12&-5/2&-5/{9}\\
(1,12,2,10)&4&15&-2/5&-12/5\\
(2,17,2,12)&7&23&-{34}/{161}&-{578}/{161}.
\end{array}
\]
Each normalized class has negative square. Therefore Lemma~\ref{lem:index}(ii) gives
\[
H^2(X,\mathcal T_{1,p+1})=0.
\]

It remains to consider the pairs $(h, k)\neq (1, p+1)$ in the three cases. For $h=1$,
\[
\begin{array}{c|c}
(p,d,r,L\cdot C)&
L\cdot\delta(\mathcal T_{1,k})\text{ for }k<p+1\\ \hline
(1,10,2,9)&L\cdot\delta(\mathcal T_{1,1})=10/3,\\
(1,12,2,10)&L\cdot\delta(\mathcal T_{1,1})=6,\\
(2,17,2,12)&L\cdot\delta(\mathcal T_{1,1})=102/7,\quad
L\cdot\delta(\mathcal T_{1,2})=1156/161.
\end{array}
\]
All these degrees are positive, so the preceding slope argument applies. For $h\ge 2$, the minimum occurs at $(h, k)=(2, p+1)$; its value is obtained from the corresponding $L\cdot\delta(\mathcal T_{1,p+1})$ by adding $L^2=2d$, and is again positive. This completes the proof.
\end{proof}

\section{The Schur complex and the proof of the main theorems}
\label{sec:mainproof}

We first recall the Schur complex associated with the presentation \eqref{eq:ABM} of $M_L$. Combining this complex with the two vanishing results in \S~\ref{sec:mlcompute}, we complete the proofs of the main theorems.

\begin{lemma}\label{lem:schurcomplex}
For every $k\geq1$, the short exact sequence \eqref{eq:ABM} induces an exact
complex
\begin{equation}\label{eq:schurcpx}
\begin{split}
0\longrightarrow S^k\mathcal A
&\longrightarrow S^{k-1}\mathcal A\otimes\mathcal B
\longrightarrow S^{k-2}\mathcal A\otimes\bigwedge^2\mathcal B
\longrightarrow\cdots\\
&\longrightarrow\mathcal A\otimes\bigwedge^{k-1}\mathcal B
\longrightarrow\bigwedge^k\mathcal B
\longrightarrow\bigwedge^kM_L\longrightarrow0.
\end{split}
\end{equation}
\end{lemma}

\begin{proof}
Over an arbitrary field, 
the Schur complex of the map $\mathcal A\to\mathcal B$ for the one-column partition $(1^k)$ has terms $\Gamma^{k-j}\mathcal A\otimes\bigwedge^j\mathcal B$; see \cite[Definition~V.1.6 and Corollary~V.1.15]{ABW82}, and also \cite[Section~2.4]{Weyman03}. Here $\Gamma^m$ denotes the $m$-th divided power.

Since the cokernel $M_L$ is locally free, the injection $\mathcal A\to\mathcal B$ splits locally. On an open set trivializing $M_L$, one has $\mathcal B\simeq\mathcal A\oplus M_L$. \cite[Corollary~V.1.15]{ABW82} applies to this split injection, implying that the Schur complex is acyclic in positive degrees, with zeroth homology $\bigwedge^kM_L$. Since exactness is local, the global complex is exact. Finally, in characteristic zero, divided and symmetric powers are canonically identified. This proves \eqref{eq:schurcpx}.
\end{proof}

\begin{proposition}\label{prop:green-vanishing}
Under the standing hypotheses \eqref{eq:stamainrange}, one has
\[
H^1\big(X,\bigwedge^kM_L\otimes L^{\otimes h}\big)=0
\]
for every $1\leq k\leq p+1$ and $h\geq1$.
\end{proposition}

\begin{proof}
Tensor \eqref{eq:schurcpx} by $L^{\otimes h}$, and put
$
K_{h,k}=\ker\left(
\bigwedge^k\mathcal B\otimes L^{\otimes h}
\longrightarrow\bigwedge^kM_L\otimes L^{\otimes h}
\right).$
The exact sequence
\[
0\longrightarrow K_{h,k}
\longrightarrow\bigwedge^k\mathcal B\otimes L^{\otimes h}
\longrightarrow\bigwedge^kM_L\otimes L^{\otimes h}
\longrightarrow0
\]
yields a long exact sequence containing
\[
H^1\big(X,\bigwedge^k\mathcal B\otimes L^{\otimes h}\big)
\longrightarrow
H^1\big(X,\bigwedge^kM_L\otimes L^{\otimes h}\big)
\longrightarrow H^2(X,K_{h,k}).
\]
The first group vanishes by Corollary~\ref{prop:B-vanishing}.

The exactness of the Schur complex gives a surjection
\[
\mathcal T_{h,k}
=\mathcal A\otimes\bigwedge^{k-1}\mathcal B\otimes L^{\otimes h}
\longrightarrow K_{h,k}.
\]
Let $R_{h,k}$ be its kernel. Taking cohomology of
\[
0\longrightarrow R_{h,k}\longrightarrow\mathcal T_{h,k}
\longrightarrow K_{h,k}\longrightarrow0
\]
induces the exact sequence
\[
\cdots\rightarrow H^2(X,\mathcal T_{h,k})\longrightarrow H^2(X,K_{h,k})
\longrightarrow H^3(X,R_{h,k})\rightarrow \cdots
\]
The first group vanishes by Proposition~\ref{prop:T-vanishing}, and the third group vanishes because $X$ is a surface. Therefore, $H^2(X,K_{h,k})=0$, which proves the claim.
\end{proof}

\begin{proof}[Proof of Theorem~\ref{main thm2}]
Parts~(1) and (2) are exactly Corollary~\ref{Ito18}; note that the corollary includes the equality case, which is not covered by the strict numerical inequality in \cite[Theorem~1.2]{Ito18}. 

For part~(3),  if $\epsilon_1(l)<1/(p+2)$, Theorem~\ref{thm:caucci} gives Property $N_p$. Otherwise the hypotheses \eqref{eq:stamainrange} hold. The construction of Sections~\ref{sec:FM}--\ref{sec:cohomology}, together with Proposition~\ref{prop:green-vanishing}, gives
\[
H^1\big(X,\bigwedge^kM_L\otimes L^{\otimes h}\big)=0
\]
for every $1\leq k\leq p+1$ and $h\geq1$. Proposition~\ref{prop:green} therefore proves Property $N_p$.
\end{proof}

\begin{proof}[Proof of Theorem~\ref{main thm}]
We first prove part~(1).  Suppose that $d\geq7$ and that $X$ contains no elliptic curve of $L$-degree at most $2$.  Proposition~\ref{p=0} gives $\epsilon_1(l)\leq1/2$. If the inequality is strict, then Theorem~\ref{thm:caucci} proves Property $N_0$.  Equality occurs precisely when
\[
L\simeq M^{\otimes2},\qquad M^2=4,
\]
and the same proposition shows that $|M|$ has no fixed component. In this case \cite[Lemma~6]{Ohbuchi93} says that $M^{\otimes2}$ is not normally generated, so $L$ does not satisfy Property $N_0$. Conversely, \cite[Lemma~2]{Ohbuchi93} shows that the fixed-component-freeness of $|M|$ is equivalent to the very ampleness of $M^{\otimes2}$.  Its restriction to every elliptic curve is therefore very ample and has degree at least $3$; thus the exceptional case does satisfy condition~(ii).

For part~(2), the implication {(ii)}$\implies${(i)} follows from Theorem~\ref{main thm2}. For (i)$\implies$(ii), under the numerical hypotheses in either part, an elliptic curve of $L$-degree at most $p+2$ (with $p=0$ in part~(1)) obstructs Property $N_p$ by \cite[Theorem~4.1 and p.~566]{KL19}. This completes the proof.
\end{proof}

\begin{proof}[Proof of Proposition~\ref{GP Conjecture}]
For every elliptic curve $E\subseteq X$, the restriction $L|_E$ is very ample, and hence $L\cdot E=\deg(L|_E)\geq3$. Thus $X$ contains no elliptic curve of $L$-degree at most $2$. Moreover, a polarization of type $(1,d)$ is primitive, so $L\simeq M^{\otimes2}$ is impossible.  Proposition~\ref{p=0} therefore gives $\epsilon_1(l)<1/2$.  In particular, Theorem~\ref{thm:caucci}, with $p=0$, shows that $L$ is projectively normal.

We next prove that its homogeneous ideal is generated in degrees at most three.  Choose a rational number $x$ with $\frac{\epsilon_1(l)}{1-\epsilon_1(l)}<x<1$. By \cite[Theorem~D, p.841]{JP20}, the $\mathbb Q$-twisted bundle $M_L\langle xl\rangle$ is $\IT(0)$.  For every $h\geq2$, one has
\[
M_L^{\otimes2}\otimes L^{\otimes h}
=\bigl(M_L\langle xl\rangle\bigr)^{\otimes2}
 \otimes\ox\langle(h-2x)l\rangle.
\]
Since $h-2x>0$, the last term is $\IT(0)$. Applying Proposition \ref{positivity of tensor}, we obtain that
\[
H^1\big(X,M_L^{\otimes2}\otimes L^{\otimes h}\big)=0
\qquad h\geq2.
\]
The specialization $p=r=1$ of \cite[Proposition~6.3(a)]{PP04} now shows that the homogeneous ideal is generated by forms of degree at most three. The embedding by a complete linear system is nondegenerate, so these are quadrics and cubics.

Finally, suppose that $d\geq10$. Theorem~\ref{main thm} (2), applied with $p=1$, says that $L$ satisfies Property $N_1$ if and only if there is no elliptic curve $E\subseteq X$ with $L\cdot E\leq3$.  Since $L$ is already projectively normal, Property $N_1$ is equivalent to its homogeneous ideal being generated by quadrics.  This proves the final assertion.
\end{proof}

\section{Examples and complements}\label{sec:examples}
Remarks in \cite[Remark~4.2]{Rojas22} and \cite[Remark~2.3]{Caucci20} indicate that no polarized abelian surface with irrational $\epsilon_1(l)$ was previously known, although it was expected.  The following example provides one; see also \cite{Ito26}.
\begin{example}[Irrational $\epsilon_1(l)$]\label{irrational threshold}
    Let $(X,L)$ be a polarized abelian surface such that
    \[\NS(X)=\mathbb{Z}[L]\oplus\mathbb{Z}[C], \quad L^2=34,\quad
    L\cdot C=12, \text{ and } C^2=4.\]
    We choose the surface so that both $L$ and $C$ are ample, as in the  construction below.  Since $h^0(X,\ox(C))=2$, the line bundle $\ox(C)$ is not globally generated. If $q$ is a base point of $|C|$,    then $h^1(\ox(C)\otimes\is q)=1$.  Let $\shf F$ be a non-split   extension of $\is q$ by $\sshf X(-C)$:
    \begin{align*}
        0\to \sshf{X}(-C)\to \shf{F}\to \is{q}\to 0.
    \end{align*}
    The Cayley--Bacharach condition in Serre's construction is precisely the choice of $q$ as a base point of $|C|$. Thus $\shf F$ is a vector    bundle with $\ch(\shf F)=(2,-C,1)$.
    
    Suppose, for contradiction, that $\shf F$ is $L$-Gieseker unstable. Taking the reflexive hull of a saturated rank-one destabilizing subsheaf gives a nonzero map $\ox(-B)\to\shf F$ with    $L\cdot B\leq6$. Since    $\hom(\sshf X(-B),\shf F)\neq0$, one has
    \[\hom(\sshf X(-B),\sshf X(-C))+\hom(\sshf X(-B),\is q)\neq0.\]
    Thus $B$ is effective, and hence $B^2\geq0$.  The Hodge index theorem gives
    \[
    36\geq(L\cdot B)^2\geq L^2B^2=34B^2,
    \]
    and the intersection form is even, so $B^2=0$.  On the other hand, the equation $(aL+bC)^2=0$ has no nonzero rational solution: solving for $b/a$ requires the square root of $(L\cdot C)^2-L^2C^2=8$, which is irrational. This    contradiction proves that $\shf F$ is $L$-Gieseker stable.
    
    Both $\shf F$ and $\shf G=\ox(-C)$ have discriminant zero. Their stability at large volume and Proposition~\ref{prop:semihomogeneousobj} show that they are stable throughout $\Stab_L(X)$.  Rotating the non-split extension gives
    \[
    0\longrightarrow\shf F\longrightarrow\is q
      \longrightarrow\ox(-C)[1]\longrightarrow0
    \]
    in the tilted heart along the numerical wall. Its radius is positive,    and the two outer terms are stable of the same phase there; hence this is an actual wall.

    We claim that it is the first actual wall.  Indeed, the absence of a nonzero isotropic class rules out rank-one walls.  For any wall of rank    $r'\geq2$, write its divisor class as $C'=aL+bC$ and put    $m'=L\cdot C'$. Proposition~\ref{prop:wallIq} and the intersection matrix give
    \[
    {m'}^2-68r'(r'-1)
      =(L\cdot C')^2-L^2{C'}^2=8b^2.
    \]
    Positivity of the wall radius forces $8b^2<17$, so $b\in\{0,\pm1\}$.  Its squared radius is
    \[
    {R'}^2=\frac{17-8b^2}{17{m'}^2}.
    \]
    If $b=\pm1$, then $m'=34a+12b>0$ gives $m'\geq12$, with equality only for $(a,b)=(0,1)$.  If $b=0$, then $m'\geq34$ and $R'=1/m'$. Thus the displayed wall has strictly maximal radius among all other possible walls. Since numerical walls for $\is q$ are nested \cite[Theorem~2.8]{LR23}, it is the first actual wall. Proposition~\ref{prop:1stwall} therefore
    gives
    \[
    \epsilon_1(l)
      =\frac{12-\sqrt{12^2-4\cdot17\cdot2}}{2\cdot17}
      =\frac{6-\sqrt2}{17}\notin\mathbb Q.
    \]

    To see that such a polarized abelian surface exists, put
\[
D=L-3C.
\]
The required intersection matrix is then equivalent to
\[
C^2=4,\qquad D^2=-2,\qquad C\cdot D=0.
\]
Let
\[
\Lambda=U_1\oplus U_2\oplus U_3\simeq U^{\oplus 3},
\]
where $U_i=\mathbb Ze_i\oplus\mathbb Zf_i$, with
$e_i^2=f_i^2=0$ and $e_i\cdot f_i=1$. The vectors
\[
c=e_1+2f_1,\qquad \delta=e_2-f_2
\]
generate a primitive sublattice
\[
M:=\mathbb Zc\oplus\mathbb Z\delta
   \simeq \langle 4\rangle\oplus\langle-2\rangle
   \subset \Lambda.
\]
Choose a very general period
\[
[\sigma]\in \Omega_M:=
\left\{
[\sigma]\in\mathbb P(M^\perp\otimes\mathbb C)
\ \middle|\
(\sigma,\sigma)=0,\ 
(\sigma,\overline\sigma)>0
\right\}.
\]
By the surjectivity of the period map for complex two-dimensional tori \cite[Theorem~II, especially pp.~56--57]{Shioda78}, the period line $[\sigma]$ is realized as $H^{2,0}(X)$ for a complex two-torus $X$, and by the Lefschetz $(1,1)$-theorem, $M\subseteq\NS(X)$. Since the condition that an additional integral class be of type $(1,1)$ defines a proper hyperplane section of the period domain, a very general choice of $[\sigma]$ gives
\[
\NS(X)=M.
\]
Denote the divisor classes corresponding to $c,\delta$ by $C,D$, respectively, and put $L=3C+D$. Then
\[
L^2=34,\qquad L\cdot C=12.
\]
Since $L^2>0$, either $L$ or $-L$ is ample. After replacing $(L, C, D)$ simultaneously by their negatives if necessary, we may assume that $L$ is ample. Thus $X$ is an abelian surface. Moreover, $C^2>0$ and $L\cdot C>0$, so $C$ is ample. Finally,
\[
\NS(X)
 =\mathbb Z[C]\oplus\mathbb Z[D]
 =\mathbb Z[L]\oplus\mathbb Z[C].
\]
\end{example}

The next example shows that when $d\le2(p+2)^2$, the inequality $\epsilon_1(l)\ge\frac1{p+2}$ can occur.

\begin{example}\label{threshold can be large}
Let $E_1$ and $E_2$ be non-isogenous elliptic curves, and let $X$ be the quotient $E_1\times E_2$ by a diagonal subgroup of order two. For a general such quotient, the images of the two curves, again denoted $E_1,E_2$, satisfy
\[
\NS(X)=\mathbb Z[E_1]\oplus\mathbb Z[E_2],\qquad
E_1^2=E_2^2=0,\qquad E_1\cdot E_2=2.
\]
Let
$L=aE_1+bE_2$, where $a,b\in\mathbb Z_{>0}$.  Then $d=\frac{L^2}{2}=2ab$, $L\cdot E_1=2b$, and $L\cdot E_2=2a$.  We make the first-wall computation explicit.   By Proposition~\ref{prop:wallIq}, any actual wall of $\is q$ is determined by numerical data
\[
\ch(\shf F)=(r,-C,r-1),
\qquad
C=xE_1+yE_2,
\]
where $r\geq1$, $x,y\in\mathbb Z_{\geq0}$, and
\begin{equation}\label{eq:quotient-wall-diophantine}
4xy=C^2=2r(r-1).
\end{equation}
When $r=1$, the class $C$ is a positive multiple of either $E_1$ or $E_2$.  When $r\geq2$, the class $C$ is ample, so $x,y>0$.

Set
\[
m=L\cdot C=2(bx+ay).
\]
By \eqref{eq:first-wall-circle}, the corresponding numerical wall is
\begin{equation*}\label{eq:quotient-wall-circle}
t^2+
\left(
s+\frac{2r-1}{2(bx+ay)}
\right)^2
=
\frac{(2r-1)^2}{4(bx+ay)^2}
-\frac1{2ab}.
\end{equation*}
Thus, let
\[
\lambda
=\frac{2r-1}{2(bx+ay)},
\]
then the squared radius is $R^2=\lambda^2-\frac1{2ab}$. Since the numerical walls of $\is q$ are nested, the first wall is the one for which $\lambda$ is maximal.

There are three extremal candidates:
\[
\begin{array}{c|c|c}
(r,C)&\lambda&R^2\\ \hline
(1,E_1)
&
\tfrac1{2b}
&
\tfrac{a-2b}{4ab^2}
\\[8pt]
(2,E_1+E_2)
&
\tfrac3{2(a+b)}
&
\tfrac{(2a-b)(2b-a)}
      {4ab(a+b)^2}
\\[8pt]
(1,E_2)
&
\tfrac1{2a}
&
\tfrac{b-2a}{4a^2b}.
\end{array}
\]
The two rank-one walls are induced by the sequences
\[
0\longrightarrow\mathcal O_X(-E_i(q))
\longrightarrow\is{q}
\longrightarrow\mathcal O_{E_i(q)}(-q)
\longrightarrow0,
\]
where $E_i(q)$ denotes the translate of $E_i$ through $q$.

The middle wall is a rank-two semihomogeneous wall with
\begin{equation*}\label{eq:quotient-rank-two-wall}
0\longrightarrow\shf G
\longrightarrow\shf F
\longrightarrow\is{q}
\longrightarrow0,
\end{equation*}
where
\[
\ch(\shf F)=(2,-E_1-E_2,1),
\qquad
\ch(\shf G)=(1,-E_1-E_2,2).
\]
Since $\shf F$ and $\shf G$ are simple semihomogeneous bundles, they are stable throughout the relevant slice by Proposition~\ref{prop:semihomogeneousobj}.  Hence this wall is actual whenever its radius is positive, namely when
\[
\frac a2<b<2a.
\]

It remains to check that no other wall is larger.  Relation \eqref{eq:quotient-wall-diophantine} implies
\begin{equation*}\label{eq:quotient-integral-estimates}
x+2y\geq2r-1,
\qquad
2x+y\geq2r-1.
\end{equation*}
For $r\geq2$, for example,
\[
(x+2y)^2\geq8xy
=4r(r-1)=(2r-1)^2-1,
\]
and the first inequality follows by integrality; the second is similar. The case $r=1$ is immediate since $C\neq0$.

If $b\leq a/2$, then $bx+ay\geq b(x+2y)\geq b(2r-1)$ and consequently $\lambda\leq\frac1{2b}$. Thus the $E_1$-wall is outermost. Similarly, if $b\geq2a$, then $\lambda\leq\frac1{2a}$, and the primitive $E_2$-wall is outermost.

Finally, suppose that $a/2\leq b\leq2a$.  In this range, $3(bx+ay)\geq(a+b)(2r-1)$. Therefore, $\lambda \leq\frac3{2(a+b)}$, with equality for $(r,C)=(2,E_1+E_2)$.  Hence the rank-two wall is the first actual wall throughout $a/2<b<2a$.

We can now apply Proposition~\ref{prop:1stwall}.  For the rank-two wall,
\[
r=2,\qquad
L\cdot (E_1+E_2)=2(a+b),
\qquad \text{ and }\qquad
\overline\Delta
=4(a+b)^2-4(2ab)\cdot2
=4(a-b)^2.
\]
Consequently, $\epsilon_1(l) =1/{\max\{a,b\}}$.

 At $b=a/2$ and $b=2a$, the maximal numerical wall has radius zero, so there is no actual wall of positive radius; the no-wall formula gives $\epsilon_1(l)=1/\sqrt d$, which agrees with the adjacent expressions.  We therefore obtain
\[\epsilon_1(l)=\left\{\begin{array}{ll}
    \frac{1}{2b}  & b\le \frac{a}{2}\\
    \frac{1}{a} &   \frac{a}{2}\le b\le a\\
     \frac{1}{b} &   a\le b\le 2a\\
     \frac{1}{2a} &   2a\le b\\ 
   \end{array}\right.\]
Fix an integer $p\geq0$, and assume that $d\geq(p+2)^2+1$ and
$L\cdot E_i\ge p+3$ for $i=1,2$.

\textbf{Claim.} For every elliptic curve $E\subseteq X$, one has
$L\cdot E\ge p+3$.

Indeed, every primitive isotropic class in this N\'eron--Severi lattice lies on one of the two rays generated by $E_1$ and $E_2$.  Thus every elliptic curve is a translate of one of them, and the claim follows from the displayed assumptions.

Nevertheless, $\epsilon_1(l)\ge\frac1{p+2}$ can still occur.  For example, if $a=b=\lceil(p+\frac52)/\sqrt2\rceil$, then
\[
\epsilon_1(l)
=\frac{1}{\lceil(p+\frac52)/\sqrt2\rceil}
=\frac{\sqrt2}{p+2}+O\big(\frac{1}{(p+2)^2}\big),
\qquad p\to\infty.
\]

\begin{figure}[ht]
\centering
\includegraphics[width=0.8\textwidth]{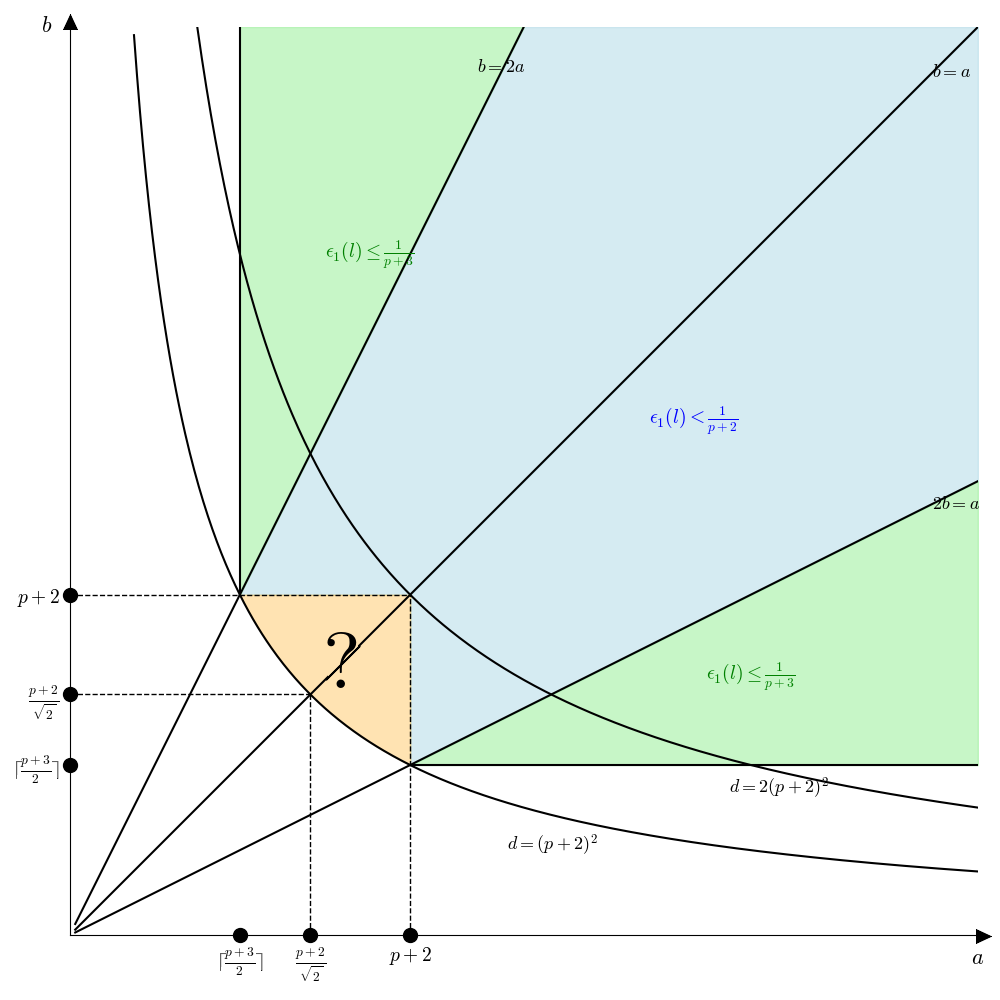}
\caption{Fix $p\ge 1$. The green (resp. blue) region represents the rank one (resp. rank two) cases, where $\epsilon_1(l)<\frac{1}{p+2}$. In the orange region, the strict inequality fails. Theorem~\ref{main thm2}(3) nevertheless proves Property $N_p$ under the elliptic-curve hypothesis.}
\label{fig:threshold-regions}
\end{figure}
\end{example}

The following example indicates that the lower bound on $d$ in Theorem~\ref{main thm} is sharp for $p=0$ and $1$.

\begin{example}\label{sharpness for p=0, 1}
A very general polarization of type $(1,6)$ does not satisfy Property $N_0$, and a very general one of type $(1,9)$ does not satisfy Property $N_1$. 

For a very general polarized abelian surface of type $(1,d)$, one has $\NS(X)=\mathbb Z[L]$ and $X$ contains no elliptic curve. If $d\ge 5$, then $L$ is very ample by Reider's criterion \cite{Reider88}.

If $d=6$, then 
\[\dim\Sym^2H^0(X,L)=\binom{7}{2}<h^0(X,L^{\otimes2})=24,\]
hence $L$ fails $N_0$. 

If $d=9$, assume, for contradiction, that $L$ satisfies Property $N_1$. Put $V=H^0(X,L)$, and let $I$ be the homogeneous ideal of the embedding. Since $N_1$ implies projective normality, we have
\[
\dim I_2=\dim\Sym^2V-h^0(X,L^{\otimes 2})
=\binom{10}{2}-36=9
\]
and
\[
\dim I_3=\dim\Sym^3V-h^0(X,L^{\otimes 3})
=\binom{11}{3}-81=84.
\]
Property $N_1$ would require $I$ to be generated by quadrics.  The multiplication map $V\otimes I_2\longrightarrow I_3$ would therefore be surjective, which is impossible because
\[
\dim(V\otimes I_2)=9\cdot9=81<84=\dim I_3.
\]
Hence $L$ fails Property $N_1$.
\end{example}

We finish with a sharp upper bound for $\epsilon_1(l)$.
\begin{proposition}\label{a  rough bound  of r}
In the setting of Proposition~\ref{prop:1stwall}, assume $r\geq2$.  Then
\[
\epsilon_1(l)\leq\sqrt{\frac{r}{(r-1)d}}.
\]
The inequality is strict if and only if the $L$-discriminant $\overline\Delta>0$.  Consequently, $\epsilon_1(l)\leq\sqrt{2/d}$.
\end{proposition}
\begin{proof}
By the Hodge index theorem,
\[
L\cdot C\geq2\sqrt{r(r-1)d}.
\]
On the other hand,
\[
\epsilon_1(l)=-p^+_{\shf Q}
=
\frac{L\cdot C-\sqrt{\overline\Delta}}
     {2(r-1)d}
=
\frac{2r}{L\cdot C+\sqrt{\overline\Delta}}
\leq
\frac{2r}{2\sqrt{r(r-1)d}}
=
\sqrt{\frac{r}{(r-1)d}}.
\]
Equality holds if and only if $\overline\Delta=(L\cdot C)^2-4dr(r-1)=0$.
\end{proof}
\begin{remark}
The bound is sharp: taking $a=b$ in Example~\ref{threshold can be large} gives $r=2$ and $d=2a^2$, hence $\epsilon_1(l)=1/a=\sqrt{2/d}$.
\end{remark}

\bibliographystyle{amsalpha}
\bibliography{abelian_syzygies}

@book{AN10,
  title={Koszul cohomology and algebraic geometry},
  author={Aprodu, Marian and Nagel, Jan},
  volume={52},
  year={2010},
  publisher={American Mathematical Soc.}
}

@article{AB12,
  title={Bridgeland-stable moduli spaces for {$K$}-trivial surfaces},
  author={Arcara, Daniele and Bertram, Aaron},
  journal={Journal of the European Mathematical Society},
  volume={15},
  number={1},
  pages={1--38},
  year={2013},
  publisher={European Mathematical Society-EMS-Publishing House GmbH}
}

@article{Agostini17,
  title={A note on homogeneous ideals of abelian surfaces},
  author={Agostini, Daniele},
  journal={Bulletin of the London Mathematical Society},
  volume={49},
  number={2},
  pages={220--225},
  year={2017},
  publisher={Wiley Online Library}
}

@article {ABW82,
    AUTHOR = {Akin, Kaan and Buchsbaum, David A. and Weyman, Jerzy},
     TITLE = {Schur functors and {S}chur complexes},
   JOURNAL = {Adv. Math.},
  FJOURNAL = {Advances in Mathematics},
    VOLUME = {44},
      YEAR = {1982},
    NUMBER = {3},
     PAGES = {207--278},
       DOI = {10.1016/0001-8708(82)90039-1}
}

@article{Bayer19,
  title={A short proof of the deformation property of {B}ridgeland stability conditions},
  author={Bayer, Arend},
  journal={Mathematische Annalen},
  volume={375},
  number={3},
  pages={1597--1613},
  year={2019},
  publisher={Springer}
}

@article{BMS16,
  title={The space of stability conditions on abelian threefolds, and on some {C}alabi-{Y}au threefolds},
  author={Bayer, Arend and Macr{\`\i}, Emanuele and Stellari, Paolo},
  journal={Inventiones mathematicae},
  volume={206},
  pages={869--933},
  year={2016},
  publisher={Springer}
}

@article{Bridgeland07,
  title={Stability conditions on triangulated categories},
  author={Bridgeland, Tom},
  journal={Annals of Mathematics},
  volume={166},
  number={2},
  pages={317--345},
  year={2007},
  publisher={JSTOR}
}

@ARTICLE{Bridgeland:K3,
  author     = {Bridgeland, Tom},
  title      = {Stability conditions on {$K3$} surfaces},
  journal    = {Duke Math. J.},
  year       = 2008,
  volume     = 141,
  pages      = {241--291},
  number     = 2,
  coden      = {DUMJAO},
  doi        = {10.1215/S0012-7094-08-14122-5},
  eprint     = {arXiv:math/0307164},
  fjournal   = {Duke Mathematical Journal},
  issn       = {0012-7094},
  mrclass    = {14F05 (14J28 18E30)},
  mrnumber   = {2376815 (2009b:14030)},
  mrreviewer = {Andrei D. Halanay},
  url        = {http://dx.doi.org/10.1215/S0012-7094-08-14122-5}
}

@article {Caucci20,
    AUTHOR = {Caucci, Federico},
     TITLE = {The basepoint-freeness threshold and syzygies of abelian
              varieties},
   JOURNAL = {Algebra Number Theory},
    VOLUME = {14},
      YEAR = {2020},
    NUMBER = {4},
     PAGES = {947--960},
       DOI = {10.2140/ant.2020.14.947}
}

@article{Dell25,
  title={Stability conditions on free abelian quotients},
  author={Dell, Hannah},
  journal={{\'E}pijournal de G{\'e}om{\'e}trie Alg{\'e}brique},
  volume={9},
  year={2025},
  publisher={Episciences. org}
}

@article {FLZ:ab3,
    AUTHOR = {Fu, Lie and Li, Chunyi and Zhao, Xiaolei},
     TITLE = {Stability manifolds of varieties with finite {A}lbanese
              morphisms},
   JOURNAL = {Trans. Amer. Math. Soc.},
  FJOURNAL = {Transactions of the American Mathematical Society},
    VOLUME = {375},
      YEAR = {2022},
    NUMBER = {8},
     PAGES = {5669--5690},
      ISSN = {0002-9947,1088-6850},
   MRCLASS = {14F08 (14J60 14K05 18G80)},
  MRNUMBER = {4469233},
MRREVIEWER = {Arnav\ Tripathy},
       DOI = {10.1090/tran/8651},
       URL = {https://0-doi-org.pugwash.lib.warwick.ac.uk/10.1090/tran/8651},
}

@article{Garcia04,
  title={Projective normality of abelian surfaces of type (1, 2 d)},
  author={Garc{\'\i}a, Luis Fuentes},
  journal={manuscripta mathematica},
  volume={114},
  number={3},
  pages={385--390},
  year={2004},
  publisher={Springer}
}

@article{GP98,
  title={Equations of $(1, d)$-polarized abelian surfaces},
  author={Gross, Mark and Popescu, Sorin},
  journal={Mathematische Annalen},
  volume={310},
  number={2},
  pages={333--377},
  year={1998}
}

@article {Gulbrandsen08,
    AUTHOR = {Gulbrandsen, Martin G.},
     TITLE = {Fourier--{M}ukai transforms of line bundles on derived
              equivalent abelian varieties},
   JOURNAL = {Le Matematiche},
    VOLUME = {63},
      YEAR = {2008},
    NUMBER = {1},
     PAGES = {123--137}
}

@article {Ito18,
    AUTHOR = {Ito, Atsushi},
     TITLE = {A remark on higher syzygies on abelian surfaces},
   JOURNAL = {Comm. Algebra},
    VOLUME = {46},
      YEAR = {2018},
    NUMBER = {12},
     PAGES = {5342--5347},
       DOI = {10.1080/00927872.2018.1464172}
}

@article{Ito22,
  author  = {Ito, Atsushi},
  title   = {M-regularity of $\mathbb{Q}$-twisted sheaves and its application to linear systems on abelian varieties},
  journal = {Trans. Amer. Math. Soc.},
  year    = {2022},
  volume  = {375},
  number  = {9},
  pages   = {6653--6673},
  doi     = {10.1090/tran/8725}
}

@article{Ito23,
  title={Higher syzygies on general polarized Abelian varieties of type (1,$\cdots$, 1, d)},
  author={Ito, Atsushi},
  journal={Mathematische Nachrichten},
  volume={296},
  number={8},
  pages={3395--3410},
  year={2023},
  publisher={Wiley Online Library}
}

@article{Ito26,
  title={Remarks on basepoint-freeness thresholds of polarized abelian surfaces},
  author={Ito, Atsushi},
  journal={arXiv preprint arXiv:2605.19353},
  year={2026}
}

@article {JP20,
    AUTHOR = {Jiang, Zhi and Pareschi, Giuseppe},
     TITLE = {Cohomological rank functions on abelian varieties},
   JOURNAL = {Ann. Sci. \'{E}c. Norm. Sup\'er. (4)},
    VOLUME = {53},
      YEAR = {2020},
    NUMBER = {4},
     PAGES = {815--846},
       DOI = {10.24033/asens.2435}
}

@article {KL19,
    AUTHOR = {K\"uronya, Alex and Lozovanu, Victor},
     TITLE = {A {R}eider-type theorem for higher syzygies on abelian
              surfaces},
   JOURNAL = {Algebr. Geom.},
    VOLUME = {6},
      YEAR = {2019},
    NUMBER = {5},
     PAGES = {548--570},
       DOI = {10.14231/AG-2019-025}
}

@article {LR23,
    AUTHOR = {Lahoz, Mart\'{\i} and Rojas, Andr\'es},
     TITLE = {Chern degree functions},
   JOURNAL = {Commun. Contemp. Math.},
    VOLUME = {25},
      YEAR = {2023},
    NUMBER = {5},
     PAGES = {2250007},
       DOI = {10.1142/S0219199722500079}
}

@book {LazarsfeldI,
    AUTHOR = {Lazarsfeld, Robert},
     TITLE = {Positivity in algebraic geometry. {I}},
    SERIES = {Ergebnisse der Mathematik und ihrer Grenzgebiete. 3. Folge},
    VOLUME = {48},
 PUBLISHER = {Springer-Verlag},
   ADDRESS = {Berlin},
      YEAR = {2004},
       DOI = {10.1007/978-3-642-18808-4}
}

@article{Lazarsfeld90,
  title={Projectivit{\'e} normale des surfaces ab{\'e}liennes},
  author={Lazarsfeld, Robert},
  journal={(written by O. Debarre), preprint Europroj},
  volume={14},
  year={1990}
}

@article {Mukai78,
    AUTHOR = {Mukai, Shigeru},
     TITLE = {Semi-homogeneous vector bundles on an abelian variety},
   JOURNAL = {J. Math. Kyoto Univ.},
    VOLUME = {18},
      YEAR = {1978},
    NUMBER = {2},
     PAGES = {239--272},
       DOI = {10.1215/kjm/1250522574}
}

@article{Mukai81,
  title={Duality between {$D(X)$} and {$D(\widehat{X})$} with its application to {Picard} sheaves},
  author={Mukai, Shigeru},
  journal={Nagoya Mathematical Journal},
  volume={81},
  pages={153--175},
  year={1981},
  publisher={Cambridge University Press}
}

@book {Mumford74,
    AUTHOR = {Mumford, David},
     TITLE = {Abelian varieties},
   EDITION = {Second},
 PUBLISHER = {Oxford University Press},
   ADDRESS = {London},
      YEAR = {1974}
}

@article {Ohbuchi93,
    AUTHOR = {Ohbuchi, Akira},
     TITLE = {A note on the normal generation of ample line bundles on an
              abelian surface},
   JOURNAL = {Proc. Amer. Math. Soc.},
    VOLUME = {117},
      YEAR = {1993},
    NUMBER = {1},
     PAGES = {275--277}
}

@article{Pareschi00,
  title={Syzygies of abelian varieties},
  author={Pareschi, Giuseppe},
  journal={Journal of the American Mathematical Society},
  volume={13},
  number={3},
  pages={651--664},
  year={2000}
}

@article {PP04,
    AUTHOR = {Pareschi, Giuseppe and Popa, Mihnea},
     TITLE = {Regularity on abelian varieties. {II}. {B}asic results on
              linear series and defining equations},
   JOURNAL = {J. Algebraic Geom.},
    VOLUME = {13},
      YEAR = {2004},
    NUMBER = {1},
     PAGES = {167--193},
       DOI = {10.1090/S1056-3911-03-00345-X}
}

@inproceedings{PP11,
  author    = {Pareschi, Giuseppe and Popa, Mihnea},
  title     = {Regularity on Abelian varieties {III}: relationship with generic vanishing and applications},
  booktitle = {Grassmannians, moduli spaces and vector bundles},
  series    = {Clay Math. Proc.},
  volume    = {14},
  pages     = {141--167},
  year      = {2011},
  publisher = {Amer. Math. Soc.},
  address   = {Providence, RI},
  doi       = {10.1090/clay/014/10}
}

@article {Reider88,
    AUTHOR = {Reider, Igor},
     TITLE = {Vector bundles of rank {$2$} and linear systems on algebraic
              surfaces},
   JOURNAL = {Ann. of Math. (2)},
    VOLUME = {127},
      YEAR = {1988},
    NUMBER = {2},
     PAGES = {309--316},
       DOI = {10.2307/2007055}
}

@article{Rojas22,
  title={The basepoint-freeness threshold of a very general abelian surface},
  author={Rojas, Andr{\'e}s},
  journal={Selecta Mathematica},
  volume={28},
  number={2},
  pages={Paper No. 34, 14 pp.},
  year={2022},
  publisher={Springer}
}

@article{Shioda78,
  title={The period map of abelian surfaces},
  author={Shioda, Tetsuji},
  journal={J. Fac. Sci. Univ. Tokyo},
  volume={25},
  pages={47--59},
  year={1978}
}

@book {Weyman03,
    AUTHOR = {Weyman, Jerzy},
     TITLE = {Cohomology of vector bundles and syzygies},
    SERIES = {Cambridge Tracts in Mathematics},
    VOLUME = {149},
 PUBLISHER = {Cambridge University Press},
   ADDRESS = {Cambridge},
      YEAR = {2003},
       DOI = {10.1017/CBO9780511546556}
}

@ARTICLE{HMS:generic_K3s,
  author     = {Huybrechts, Daniel and Macr{\`{\i}}, Emanuele and Stellari, Paolo},
  title      = {Stability conditions for generic {$K3$} categories},
  journal    = {Compos. Math.},
  year       = 2008,
  volume     = 144,
  pages      = {134--162},
  number     = 1,
  doi        = {10.1112/S0010437X07003065},
  eprint     = {arXiv:math/0608430},
  fjournal   = {Compositio Mathematica},
  issn       = {0010-437X},
  mrclass    = {14J28 (14F05 18E30)},
  mrnumber   = {2388559 (2009a:14050)},
  mrreviewer = {Justin Sawon},
  url        = {http://dx.doi.org/10.1112/S0010437X07003065}
}

@article{Happel-al:tilting,
  author        = {Happel, Dieter and Reiten, Idun and Smal{\o}, Sverre O.},
  title         = {Tilting in abelian categories and quasitilted algebras},
  journal       = {Mem. Amer. Math. Soc.},
  fjournal      = {Memoirs of the American Mathematical Society},
  volume        = {120},
  number        = {575},
  pages         = {viii+88},
  year          = {1996},
  doi           = {10.1090/memo/0575},
  eprint        = {},
  archivePrefix = {},
}

\bigskip
{\small\noindent\textsc{Mathematics Institute, University of Warwick\\
Coventry CV4 7AL, United Kingdom}\\
\emph{E-mail address}:  \texttt{c.li.25@warwick.ac.uk}\par}

\bigskip
{\small\noindent\textsc{School of Mathematics, Sun Yat-sen University\\
W. 135 Xingang Rd., Guangzhou, Guangdong 510275, P.R.~China}\\
\emph{E-mail address}:  \texttt{songlei3@mail.sysu.edu.cn}\par}

\bigskip
{\small\noindent\textsc{Institute for Theoretical Sciences, Westlake University\\
No. 600 Dunyu Rd, Xihu District, Hangzhou, Zhejiang 310030, P.R.~China}\\
\emph{E-mail address}:  \texttt{wangxiao24@westlake.edu.cn}\par}

\end{document}